\newif\ifJOTAformat
\IfFileExists{svjour3.cls}{%
  \documentclass[smallextended,referee,envcountsect,envcountsame]{svjour3}
  \JOTAformattrue
}{%
  \documentclass[11pt,a4paper]{article}
  \JOTAformatfalse
}
\usepackage[T1]{fontenc}
\usepackage{lmodern}
\usepackage{amsmath,amssymb,mathtools}
\ifJOTAformat
  \smartqed
\else
\usepackage[margin=28mm]{geometry}
\usepackage{amsthm}
\fi
\usepackage{booktabs,array}
\usepackage{xcolor}
\usepackage[colorlinks=true,linkcolor=blue,citecolor=blue,urlcolor=blue]{hyperref}
\usepackage{microtype}
\ifJOTAformat
\else
\newtheorem{theorem}{Theorem}[section]
\newtheorem{proposition}[theorem]{Proposition}
\newtheorem{lemma}[theorem]{Lemma}
\newtheorem{corollary}[theorem]{Corollary}
\newtheorem{example}[theorem]{Example}
\theoremstyle{definition}
\newtheorem{definition}[theorem]{Definition}
\newtheoremstyle{jotaremark}{\topsep}{\topsep}{\normalfont}{}
  {\itshape}{.}{.5em}{}
\theoremstyle{jotaremark}
\newtheorem{remark}[theorem]{Remark}
\fi
\newcommand{\R}{\mathbb R}
\newcommand{\norm}[1]{\left\|#1\right\|}
\newcommand{\ip}[2]{\langle #1,#2\rangle}
\newcommand{\cl}{\operatorname{cl}}
\newcommand{\sgn}{\operatorname{sign}}
\newcommand{\argmin}{\operatorname*{arg\,min}}
\newcommand{\gph}{\operatorname{gph}}
\newcommand{\supp}{\operatorname{supp}}

\title{A Variational Geometric Approach to Exact Recovery by
\texorpdfstring{$\ell_1/\ell_2$}{l1/l2} Minimization}
\author{Yulan Liu \quad and\quad Rongrong Lin\\ Guangdong University of Technology}
\date{}

\begin{document}
\maketitle
\begin{abstract}
For affine-constrained $\ell_1/\ell_2$ minimization with nonzero
data, a variational geometric approach is developed to characterize
exact recovery.
Necessary and sufficient conditions for local and global exact
recovery are established in terms of directional derivatives and
cones associated with sublevel sets of the ratio objective.
The local analysis further reveals that local
minimality, first-order growth, and tilt stability are equivalent. These pointwise
characterizations lead to improved sufficient conditions for uniform
recovery of sparse signals. The sufficient condition for
uniform local recovery has coefficient
$c_s=\max\{1,(\sqrt{s}-1)/2\}$ for $s\ge2$.
For each fixed $s$, the coefficient $c_s$ is optimal within the
considered family of local null-space conditions. Projection estimates yield
a scalar bound for uniform global recovery with coefficient two, improving the earlier
coefficient three. The coefficient two is optimal within the considered
family of global null-space bounds when both $s$ and $n$ may vary.
A positive margin in the local null-space condition provides explicit growth moduli, localization radii, and bounds on linear objective perturbations that leave the localized minimizer
fixed. Examples show that uniform local recovery does not imply
uniform global recovery.
\end{abstract}


\section{Introduction}\label{Sec_intro}
Sparse recovery concerns the reconstruction of a sparse signal from
underdetermined linear measurements. A direct formulation minimizes
the $\ell_0$ objective, which counts the nonzero entries of a signal,
subject to the measurement constraints. The combinatorial nature of
this problem motivates the use of alternative sparsity-promoting
objectives. These include the convex $\ell_1$ norm \cite{FR},
the separable nonconvex $\ell_p^p$ penalty with $0<p<1$ \cite{FL},  the nonseparable difference $\ell_1-\ell_2$ \cite{YLHX}, and the \(\ell_1/\ell_2\) ratio \cite{JiaPrater2025,TaoZhangZhao2026}.
For a chosen objective, a fundamental question is to identify
conditions under which the original sparse signal is the unique
minimizer of the resulting optimization problem.

Given a matrix $A\in\R^{m\times n}$ and a vector
$b\in\R^m\setminus\{0\}$, this paper focuses on exact recovery for the
$\ell_1/\ell_2$ minimization problem
\begin{equation}\label{Eq_ratio}
 \min_{x\in\R^n}\ \Phi(x):=\frac{\norm{x}_1}{\norm{x}_2}
 \quad\text{subject to}\quad Ax=b.
\end{equation}
Throughout the paper, $\Delta:=\{x\in\R^n:Ax=b\}$ denotes the
feasible set of \eqref{Eq_ratio}, which is assumed to be nonempty.
The assumption $b\ne0$ ensures
that every feasible point is nonzero and hence that $\Phi$ is defined
on $\Delta$. When $b=0$, every nonzero feasible point $x$ has
other feasible points $tx$ with the same objective value, since
$A(tx)=0$ and $\Phi(tx)=\Phi(x)$ for $t>0$.
Taking $t\ne1$ arbitrarily close to $1$ shows that $x$ cannot be
the unique minimizer even in a neighborhood. Thus neither local
nor global exact recovery is possible for nonzero feasible signals
when $b=0$.
Here $\norm{\cdot}_1$ and $\norm{\cdot}_2$ are the
$\ell_1$ and Euclidean norms, respectively. The ratio satisfies
\[
 \Phi(tx)=\Phi(x)\quad(t\ne0),\qquad
 1\le\Phi(x)\le\sqrt{\norm{x}_0}\quad(x\ne0),
\]
where $\norm{x}_0$ is the number of nonzero entries of $x$.
The ratio is used to promote sparsity in recovery problems with
coherent measurement matrices \cite{YEX,RWDL} and in portfolio
optimization \cite{WSGAZ}. Its scale invariance makes the objective
independent of the magnitude of the signal.
Global exact recovery of a signal $x^*\in\Delta$ means that $x^*$
is the unique global minimizer of \eqref{Eq_ratio}. Local exact
recovery means that $x^*$ is the unique minimizer in some feasible
neighborhood.

For an integer $s\ge1$, a signal is called $s$-sparse if it has at
most $s$ nonzero entries. Uniform exact recovery of order $s$ means
that every $s$-sparse signal is uniquely recovered. For the ratio
model, this requirement is restricted to signals with nonzero
measurements. For $\ell_1$
minimization, this property is characterized by the null-space
property (NSP) \cite{GN,FR}. Recovery theory also extends to more
general separable penalties. Gribonval and Nielsen~\cite{GN07}
give sufficient conditions under which the same sparse representation
is uniquely recovered by every penalty in a class of admissible
separable sparsity penalties. For penalties of the form
$R(x)=\sum_i\varphi(|x_i|)$, where
$\varphi:[0,+\infty)\to[0,+\infty)$ is nondecreasing and concave
with $\varphi(0)=0$, a generalized NSP (gNSP) gives a necessary and
sufficient condition for uniform exact recovery \cite{TW}.
This condition compares the penalty contributions of a nonzero
null-space vector on a support and its complement. A complementary
approach derives sufficient recovery conditions from the restricted
isometry property (RIP). For $\ell_1$ minimization, RIP bounds also
give error estimates for noisy measurements and approximately sparse
signals \cite{CRT,FR}. For $\ell_p^p$ minimization with $0<p<1$,
RIP-based recovery guarantees and optimal RIP bounds have been
established in \cite{FL,ZL}.

For nonseparable penalties, interactions among coordinates require
recovery conditions beyond a direct extension of the separable case.
Yin et al.~\cite{YLHX} establish RIP-based exact and stable recovery
guarantees for constrained $\ell_1-\ell_2$ minimization.
Tran and Webster~\cite{TW} develop a more general framework for
symmetric nonseparable penalties concave on the nonnegative orthant.
Under additional assumptions on the penalty, the standard NSP
guarantees uniform exact recovery. Stronger assumptions allow a
relaxed NSP, with equal-height sparse signals excluded from the
recovery guarantee. Necessary and sufficient recovery conditions
have also been obtained for specific nonseparable models.
Bi and Tang~\cite{BT} characterize pointwise local, uniform local,
and uniform global exact recovery via $\ell_1-\ell_2$ minimization.
Xie et al.~\cite{XLL} extend the local characterizations to
$\ell_1-\alpha\ell_2$ minimization with $0<\alpha\le2$, both at a
fixed signal and uniformly over sparse signals. The ratio objective
in \eqref{Eq_ratio} is nonseparable and is not concave on the
nonnegative orthant, so it falls outside the concavity-based
framework of Tran and Webster.

Several recovery conditions have been established for
\eqref{Eq_ratio}. For nonnegative signals, Yin et al.~\cite{YEX}
give a sufficient condition for the unique sparsest feasible signal
to be the unique global minimizer of the ratio model. Their condition
involves the sparsity and amplitude uniformity of competing locally
sparse feasible solutions. For arbitrary signs,
Rahimi et al.~\cite{RWDL} prove local recovery under a strong
NSP. Xu et al.~\cite{XNTW} give sufficient conditions for local exact
recovery of a fixed signal and for uniform local and global exact
recovery. Their Remark~3.5 also discusses how the sufficient condition
for uniform local recovery can be weakened.
Tao and Zhang~\cite{TZ} prove that the support submatrix of every
local minimizer has full column rank. They also extend the nonnegative
global recovery result to signals of arbitrary signs, giving a
sufficient condition for recovery of the unique sparsest feasible
signal in terms of the sparsity and amplitude uniformity of competing
feasible solutions.

Recovery conditions have also been studied for more general norm-ratio
models. Zhou and Yu~\cite{ZY} establish sufficient conditions for
uniform exact recovery by $\ell_1/\ell_q$ minimization with
$1<q\le\infty$, including the case $q=2$.
Xie et al.~\cite{XSG} derive RIP-based sufficient recovery conditions
for $\ell_1/\ell_q$ minimization with $q>1$.
For $\ell_p^p/\ell_q^p$ minimization with $0<p\le1$ and $q>1$,
Yu and Huang~\cite{YH} establish sufficient conditions for local
and uniform global exact recovery, and also study noisy recovery and
numerical algorithms. Their model includes \eqref{Eq_ratio} as the
special case $p=1$ and $q=2$, so their recovery conditions provide a
direct point of comparison for the present study.

The existing analyses of the ratio model provide sufficient recovery
conditions and necessary structural properties of local minimizers.
However, a sufficient recovery condition does not, in general,
characterize whether a given signal is recovered. For local
recovery in \eqref{Eq_ratio},
a key issue is the behavior along feasible directions with zero
directional derivative, which cannot be determined by first-order
nonnegativity alone. For global recovery, comparisons over the
entire feasible set are required. These considerations motivate
the search for pointwise characterizations of local and global
exact recovery, which can in turn be used to derive improved
sufficient conditions for uniform recovery.

Variational geometry offers a natural approach to these questions.
Descent cones characterize unique recovery in convex optimization
\cite{ALMT}, and level-set geometry relates exact recovery to
objective growth in nuclear-norm optimization \cite{LPB}.
Motivated by these connections, this paper focuses on
\eqref{Eq_ratio} and develops pointwise recovery characterizations
through directional derivatives, the curvature along feasible rays,
and the geometry of sublevel sets. Relative to the broader norm-ratio
analysis in~\cite{YH}, the aim is to obtain exact local and global
recovery criteria for this specific objective and to determine their
implications for growth, stability, and uniform recovery.
The recovery conditions in~\cite{YH} are derived through norm
estimates and null-space inequalities. The present analysis first
establishes pointwise geometric characterizations and then derives
uniform sufficient conditions.

The main results of this paper are summarized as follows. First,
necessary and sufficient conditions for local exact recovery are
established in terms of directional derivatives and the tangent
cone to the reference sublevel set. A key property is that every
nonzero feasible direction with zero directional derivative has
strictly negative ray curvature. This yields the equivalence of
local minimality, first-order growth, and tilt stability.
The local recovery criterion can also be checked by
solving at most $2n$ linear programming feasibility problems, yielding
a polynomial-time test for a given signal when the matrix and signal
have rational entries.
For global exact recovery, an explicit description of the global
descent cone is obtained. Combined with the condition that this
cone intersects the null space of $A$ only at zero, this description
yields a pointwise recovery test.

These pointwise characterizations yield a sufficient condition for
uniform local recovery depending only on the null space and the
sparsity order. For $s\ge2$, its coefficient is
$c_s=\max\{1,(\sqrt{s}-1)/2\}$, improving
$(\sqrt{s}+3)/2$ in~\cite[Corollary 3.3]{YH} when $p=1$ and $q=2$.
The condition is also strictly weaker than the other local sufficient
conditions compared in this paper. For each fixed $s$, the coefficient
$c_s$ is optimal within the considered scalar family, independently
of the dimension and the measurement matrix. A positive margin in the
condition provides explicit growth moduli, localization radii, and
thresholds below which linear objective perturbations leave the
localized minimizer fixed.

For uniform global recovery, projection estimates yield a weaker
sufficient null-space condition than the compared condition of
Xu et al.~\cite{XNTW}. The resulting scalar bound improves the
coefficient three in~\cite[Theorem 4.1]{XNTW} and in the
$p=1$, $q=2$ case of~\cite[Theorem 4.3]{YH} to two. This coefficient
is optimal within the considered scalar family when both the sparsity
order and the dimension may vary. Examples establish that these
sufficient conditions are not necessary and that uniform local
recovery does not imply uniform global recovery.

The paper is organized as follows. Section~\ref{Sec_geometry}
develops the directional and sublevel-set geometry of $\Phi$ after
introducing notation and basic definitions.
Section~\ref{Sec_pointwise} establishes exact recovery and stability
criteria at a fixed signal. Section~\ref{Sec_uniform} derives uniform
recovery conditions, coefficient optimality results, and quantitative
stability bounds.

\section[Variational Geometry of the Ratio Objective]{Variational Geometry of the Ratio Objective}\label{Sec_geometry}
This section introduces notation and recalls basic concepts from
variational analysis. It then examines the directional behavior
and sublevel-set geometry of $\Phi$, which form the basis of the
recovery conditions developed in the subsequent sections.

The Euclidean inner product is denoted by $\ip\cdot\cdot$, and
the closed Euclidean unit ball by $\mathbb B$. For a matrix
$A\in\R^{m\times n}$, $\ker(A)$, $\operatorname{range}(A)$,
and $A^\top$ denote its null space, range, and transpose,
respectively. The orthogonal complement of a subspace $V\subseteq\R^n$
is denoted by $V^\perp$. For an index set $T\subseteq\{1,\ldots,n\}$,
write $T^c:=\{1,\ldots,n\}\setminus T$, and let $x_T$ and $A_T$
denote the subvector of $x$ and the column submatrix of $A$ indexed
by $T$, respectively. The support of $x\in\R^n$ is
$\supp(x):=\{i:x_i\ne0\}$. Throughout the paper, $S$ denotes
the support of the reference signal $x^*$, that is, $S:=\supp(x^*)$.
The sign function $\sgn(t)$ takes the values $1$, $0$, and $-1$
according as $t>0$, $t=0$, and $t<0$, respectively, and is applied
componentwise to vectors.
The gradient and Hessian of a smooth scalar function $g$ are
denoted by $\nabla g$ and $\nabla^2g$, respectively.
The one-sided directional derivative of $\Phi$ at $x$ in direction
$d$ is denoted by $\Phi'(x;d)$.

For a nonempty set $\Omega\subseteq\R^n$, $\cl\Omega$ denotes its
closure, $\operatorname{dist}(x,\Omega):=\inf_{y\in\Omega}\norm{x-y}_2$
denotes the distance from $x$ to $\Omega$, and $\delta_\Omega$
denotes its indicator function, equal to zero on $\Omega$ and
$+\infty$ elsewhere. The contingent tangent cone to $\Omega$ at
$\bar x\in\Omega$ is defined by \cite[Definition 6.1]{RW}
\[
 \mathcal T_\Omega(\bar x)
 :=\{d\in\R^n:\exists\,t_k\downarrow0,\ d_k\to d,
                   \ \bar x+t_kd_k\in\Omega\}.
\]
The regular and limiting normal cones to $\Omega$ at $\bar x$
are defined, respectively, by \cite[Definition 6.3]{RW}
\[
 \widehat{\mathcal N}_\Omega(\bar x)
 :=\left\{v\in\R^n:
 \limsup_{\substack{x\to\bar x,\ x\in\Omega\\x\ne\bar x}}
 \frac{\ip v{x-\bar x}}{\norm{x-\bar x}_2}\le0\right\}
\]
and
\[
 \mathcal N_\Omega(\bar x)
 :=\{v\in\R^n:\exists\,x_k\in\Omega,\ x_k\to\bar x,
       \ v_k\in\widehat{\mathcal N}_\Omega(x_k),\ v_k\to v\}.
\]
When $\Omega$ is convex, both cones coincide with the normal cone
of convex analysis \cite[Theorem 6.9]{RW}. A closed set $\Omega$
is said to be Clarke regular at $\bar x$ if
$\widehat{\mathcal N}_\Omega(\bar x)=\mathcal N_\Omega(\bar x)$
\cite[Definition 6.4]{RW}.

An extended-real-valued function $g:\R^n\to(-\infty,+\infty]$ is
proper if its domain $\operatorname{dom}g:=\{x:g(x)<+\infty\}$
is nonempty, and closed if it is lower semicontinuous.
For a proper closed function $g$ and $\bar x\in\operatorname{dom}g$,
the regular and limiting subdifferentials of $g$ at $\bar x$ are
defined, respectively, by \cite[Definition 8.3]{RW}
\[
 \widehat\partial g(\bar x)
 :=\left\{v\in\R^n:
 \liminf_{\substack{x\to\bar x\\x\ne\bar x}}
 \frac{g(x)-g(\bar x)-\ip v{x-\bar x}}{\norm{x-\bar x}_2}
 \ge0\right\}
\]
and
\[
 \partial g(\bar x)
 :=\{v\in\R^n:\exists\,x_k\xrightarrow{g}\bar x,
       \ v_k\in\widehat\partial g(x_k),\ v_k\to v\},
\]
where $x_k\xrightarrow{g}\bar x$ means that $x_k\to\bar x$ and
$g(x_k)\to g(\bar x)$. Both subdifferentials are empty outside
$\operatorname{dom}g$. If $g$ is continuously differentiable near
$\bar x$, both reduce to $\{\nabla g(\bar x)\}$
\cite[Exercise 8.8]{RW}. If $g$ is convex,
both coincide with its convex subdifferential
\cite[Proposition 8.12]{RW}.

For a closed set $\Omega$, the normal cones and subdifferentials
are related by
$\widehat{\mathcal N}_\Omega=\widehat\partial\delta_\Omega$ and
$\mathcal N_\Omega=\partial\delta_\Omega$
\cite[Exercise 8.14]{RW}. A proper closed function $g$ is said to
be subdifferentially regular at $\bar x\in\operatorname{dom}g$ if
its epigraph is Clarke regular at $(\bar x,g(\bar x))$
\cite[Definition 7.25]{RW}. The horizon subdifferential is denoted
by $\partial^\infty g$.
In particular, local Lipschitz continuity
of $g$ at $\bar x$ implies $\partial^\infty g(\bar x)=\{0\}$
\cite[Theorem 9.13]{RW}.

For a set-valued mapping $M$, its graph is denoted by
$\gph M:=\{(p,x):x\in M(p)\}$. The graphical derivative of $M$ at
$(\bar p,\bar x)\in\gph M$ is defined by
\[
 DM(\bar p\mid\bar x)(u)
 :=\{d:(u,d)\in\mathcal T_{\gph M}(\bar p,\bar x)\}.
\]

To formulate stationarity for \eqref{Eq_ratio}, define the
constrained objective $f(x):=\Phi(x)+\delta_\Delta(x)$ for $x\ne0$,
and set $f(0):=+\infty$.
\begin{definition}\label{Def_critical}
A point $x^*\in\Delta$ is called a \emph{critical point} of
\eqref{Eq_ratio} if $0\in\partial f(x^*)$.
\end{definition}

\subsection{Characterizations of Critical Points}\label{Sec_critical}

The following lemma combines the subdifferential formula for $\Phi$
in \cite[Eq.~(2.3)]{ZYP} with the normal cone to the affine set
$\Delta$ to obtain an explicit expression for $\partial f(x^*)$.

\begin{lemma}\label{Prop_subdiff}
For every $x^*\in\Delta$,
$\widehat{ \partial} f(x^*)=\partial f(x^*)=\partial\Phi(x^*)+\mathcal N_\Delta(x^*)$,
where
\begin{equation}\label{Eq_Phi}
 \widehat{\partial}\Phi(x^*)=\partial\Phi(x^*)=
 \left\{\frac{z}{\norm{x^*}_2}
 -\frac{\norm{x^*}_1x^*}{\norm{x^*}_2^3}:
 z\in\R^n,\ z_S=\sgn(x^*_S),\ |z_i|\le1\ (i\in S^c)\right\}
\end{equation}
and
$ \widehat{\mathcal N }_\Delta(x^*)=\mathcal N_\Delta(x^*)=(\ker(A))^\perp
 =\{A^\top\lambda:\lambda\in\R^m\}$.
\end{lemma}
\begin{proof}
By \cite[Eq.~(2.3)]{ZYP}, $\Phi$ is subdifferentially regular
at $x^*\ne0$, with
\[
 \partial\Phi(x^*)=
 \frac1{\norm{x^*}_2}\partial\norm{x^*}_1
 -\frac{\norm{x^*}_1}{\norm{x^*}_2^3}x^*.
\]
Substituting the standard $\ell_1$ subdifferential gives
\eqref{Eq_Phi}.

For the affine set $\Delta=x^*+\ker(A)$, the regular and limiting
normal cones coincide \cite[Theorem 6.9]{RW}. A vector $w$ belongs
to this cone if and only if $\ip wd\le0$ for all $d\in\ker(A)$.
Applying the inequality to $d$ and $-d$ gives $\ip wd=0$, so
$\mathcal N_\Delta(x^*)=(\ker(A))^\perp=\operatorname{range}(A^\top)$.

Finally, local Lipschitz continuity of $\Phi$ near $x^*$ implies
$\partial^\infty\Phi(x^*)=\{0\}$ \cite[Theorem 9.13]{RW},
so the qualification condition in \cite[Corollary 10.9]{RW}
is satisfied. Since both $\Phi$ and $\delta_\Delta$ are regular
at $x^*$, the sum rule gives $\widehat{ \partial} f(x^*)=\partial f(x^*)=\partial\Phi(x^*)+\mathcal N_\Delta(x^*)$.
\end{proof}

The next lemma gives the directional derivative and a local
expansion of $\Phi$.
\begin{lemma}\label{Prop_derivative}
Let $x^*\ne0$. For every $d\in\R^n$, the directional derivative
of $\Phi$ at $x^*$ is
\begin{equation}\label{Eq_derivative}
\begin{aligned}
 \Phi'(x^*;d)
 &:=\lim_{t\downarrow0}\frac{\Phi(x^*+td)-\Phi(x^*)}{t}\\
 &=\frac{\ip{\sgn(x^*_S)}{d_S}+\norm{d_{S^c}}_1}{\norm{x^*}_2}
 -\frac{\norm{x^*}_1\ip{x^*}{d}}{\norm{x^*}_2^3}.
\end{aligned}
\end{equation}
As $d\to0$,
\begin{equation}\label{Eq_expansion}
 \Phi(x^*+d)=\Phi(x^*)+\Phi'(x^*;d)+O(\norm{d}_2^2).
\end{equation}
Here $O(\norm d_2^2)$ denotes a remainder bounded in absolute value
by a constant multiple of $\norm d_2^2$ as $d\to0$.
The map $d\mapsto\Phi'(x^*;d)$ is continuous, convex, and positively
homogeneous.
\end{lemma}
\begin{proof}
For sufficiently small $d$, the coordinates on $S$ retain their signs,
and hence
\begin{align}\label{diff_l1}
 \norm{x^*+d}_1
 =\norm{x^*}_1+\ip{\sgn(x^*_S)}{d_S}+\norm{d_{S^c}}_1.
\end{align}
Taylor expansion of $1/\|x\|_2$ at $x^*\ne0$ gives
\[
 \frac1{\norm{x^*+d}_2}
 =\frac1{\norm{x^*}_2}
 -\frac{\ip{x^*}{d}}{\norm{x^*}_2^3}+O(\norm d_2^2).
\]
Multiplying these expressions and replacing $d$ by $td$ gives
\[
\begin{aligned}
 \frac{\Phi(x^*+td)-\Phi(x^*)}{t}
=\frac{\ip{\sgn(x^*_S)}{d_S}+\norm{d_{S^c}}_1}{\norm{x^*}_2}-\frac{\norm{x^*}_1\ip{x^*}{d}}{\norm{x^*}_2^3}
 +O(t\norm d_2^2).
\end{aligned}
\]
For fixed $d$, letting $t\downarrow0$ proves \eqref{Eq_derivative}.
The same product expansion yields \eqref{Eq_expansion}.
Continuity, convexity, and positive homogeneity follow from
\eqref{Eq_derivative}.
\end{proof}

When $\Phi'(x^*;d)=0$, the behavior along the ray depends on its
curvature. The next lemma shows that $\Phi$ decreases for small
positive steps unless $d$ is parallel to $x^*$.
\begin{lemma}\label{Lem_ray}
Let $x^*\ne0$ and $d\in\R^n$.
If $\Phi'(x^*;d)=0$, then
the second-order directional
derivative is given by
\begin{equation}\label{Eq_curvature}
 \Phi''(x^*;d):=
 \lim_{t\downarrow0}
 \frac{2\bigl(\Phi(x^*+td)-\Phi(x^*)-t\Phi'(x^*;d)\bigr)}{t^2}
 =\frac{\norm{x^*}_1}{\norm{x^*}_2^5}
 \bigl(\ip{x^*}{d}^2-\norm{x^*}_2^2\norm d_2^2\bigr)\le0.
\end{equation}
The inequality is strict if $d\notin\operatorname{span}\{x^*\}$.
If $d\in\operatorname{span}\{x^*\}$, then
$\Phi(x^*+td)=\Phi(x^*)$ for all sufficiently small $t\ge0$.
\end{lemma}
\begin{proof}
By \eqref{Eq_derivative}, the assumption $\Phi'(x^*;d)=0$ gives
\[
 \ip{\sgn(x^*_S)}{d_S}+\norm{d_{S^c}}_1
 =\frac{\norm{x^*}_1}{\norm{x^*}_2^2}\ip{x^*}{d}.
\]
Together with \eqref{diff_l1}, this identity gives, for
sufficiently small $t\ge0$,
\[
 \Phi(x^*+td)=\Phi(x^*)
 \frac{1+\dfrac{\ip{x^*}{d}}{\norm{x^*}_2^2}t}
 {\left(1+2\dfrac{\ip{x^*}{d}}{\norm{x^*}_2^2}t
 +\dfrac{\norm d_2^2}{\norm{x^*}_2^2}t^2\right)^{1/2}}.
\]
Using $(1+u)^{-1/2}=1-u/2+3u^2/8+o(u^2)$ as $u\to0$
and multiplying by the numerator gives
\[
 \Phi(x^*+td)=\Phi(x^*)
 +\frac{\norm{x^*}_1}{2\norm{x^*}_2^5}
 \bigl(\ip{x^*}{d}^2-\norm{x^*}_2^2\norm d_2^2\bigr)t^2
 +o(t^2).
\]
The definition of $\Phi''(x^*;d)$ now gives \eqref{Eq_curvature}.
The Cauchy--Schwarz inequality gives nonpositivity, with strict inequality precisely
when $d\notin\operatorname{span}\{x^*\}$.
Finally, if $d=\lambda x^*$, scale invariance gives
$\Phi(x^*+td)=\Phi((1+t\lambda)x^*)=\Phi(x^*)$
for all sufficiently small $t\ge0$.
\end{proof}

\begin{remark}\label{Rem_feasible_curvature}
The condition $Ax^*=b\ne0$ for $x^*\in\Delta$ implies
\begin{equation}\label{Eq_no_parallel_feasible}
 \ker(A)\cap\operatorname{span}\{x^*\}=\{0\}.
\end{equation}
Indeed, $A(\lambda x^*)=\lambda b=0$ forces $\lambda=0$.
By Lemma~\ref{Lem_ray}, every nonzero feasible direction $d$ with
$\Phi'(x^*;d)=0$ therefore satisfies $\Phi''(x^*;d)<0$.
\end{remark}

Lemmas~\ref{Prop_subdiff} and~\ref{Prop_derivative} give three
equivalent tests for criticality: a directional condition, a
null-space inequality, and a multiplier condition.
\begin{proposition}\label{Prop_stationary}
Let $x^*\in\Delta$. The following statements are equivalent:
\begin{itemize}
\item[(i)] $x^*$ is a critical point.
\item[(ii)] $\Phi'(x^*;d)\ge0$ for every $d\in\ker(A)$.
\item[(iii)] For every $d\in\ker(A)$,
$ \norm{d_{S^c}}_1\ge
 \left|\ip{\sgn(x^*_S)}{d_S}
 -\frac{\norm{x^*}_1}{\norm{x^*}_2^2}\ip{x^*}{d}\right|$.
 
\item[(iv)] There exist $z\in\R^n$ and $\lambda\in\R^m$ such  that
\[
 z_S=\sgn(x^*_S),\ \ |z_i|\le 1\,(i\in S^c),  {\ \ and \ \  }\frac{z}{\norm{x^*}_2}
 -\frac{\norm{x^*}_1x^*}{\norm{x^*}_2^3}
 +A^\top\lambda=0.
\]
\end{itemize}

\end{proposition}
\begin{proof}
By Definition~\ref{Def_critical} and Lemma~\ref{Prop_subdiff},
$0\in\partial f(x^*)$ holds if and only if the vectors $z$ and
$\lambda$ in (iv) exist. Thus (i) and (iv) are equivalent.

To prove (ii) $\Longleftrightarrow$ (iv), consider the continuous
convex function
\[
 H(x)=\frac{\norm{x}_1}{\norm{x^*}_2}
 -\frac{\norm{x^*}_1}{\norm{x^*}_2^3}\ip{x^*}{x},
 \qquad x\in\R^n.
\]
By \eqref {Eq_derivative},
$H'(x^*;d)=\Phi'(x^*;d)$ for every $d\in\R^n$.
Convexity of $H$ and $\Delta=x^*+\ker(A)$ show that (ii) holds
if and only if $x^*$ minimizes $H$ on $\Delta$, or equivalently,
\[
 0\in\partial H(x^*)+\mathcal N_\Delta(x^*).
\]
Since $\partial H(x^*)=\partial\Phi(x^*)$, this is (iv).

Finally, \eqref{Eq_derivative} shows that (ii) is equivalent to
\[
 \ip{\sgn(x^*_S)}{d_S}
 -\frac{\norm{x^*}_1}{\norm{x^*}_2^2}\ip{x^*}{d}
 +\norm{d_{S^c}}_1\ge0
 \qquad(d\in\ker(A)).
\]
Since $\ker(A)$ is a subspace, the inequality also holds with $d$
replaced by $-d$. The linear terms change sign, whereas
$\norm{(-d)_{S^c}}_1=\norm{d_{S^c}}_1$. Combining the two
inequalities gives
\[
 -\norm{d_{S^c}}_1
 \le \ip{\sgn(x^*_S)}{d_S}
 -\frac{\norm{x^*}_1}{\norm{x^*}_2^2}\ip{x^*}{d}
 \le \norm{d_{S^c}}_1,
\]
which is precisely (iii). Conversely, (iii) bounds the linear
expression below by $-\norm{d_{S^c}}_1$, so
\eqref{Eq_derivative} gives $\Phi'(x^*;d)\ge0$ for every
$d\in\ker(A)$. Thus (ii) and (iii) are equivalent.
\end{proof}

\subsection{Cone Geometry of \texorpdfstring{$C$}{C}}\label{Sub_cones}
Let $x^*\ne0$ and denote its reference sublevel set by
\begin{equation}\label{Eq_sublevel}
 C=\{x\in\R^n\setminus\{0\}:\Phi(x)\le\Phi(x^*)\}.
\end{equation}
The \emph{radial cone} to $C$ at $x^*$ is defined by
\[
 \mathcal R_C(x^*)=
 \{d\in \mathbb{R}^n :\exists\,\varepsilon>0,\ x^*+td\in C
                  \text{ for all }0<t<\varepsilon\}.
\]
The \emph{global descent cone} of $\Phi$ at $x^*$ is defined by
\begin{equation}\label{Eq_descent}
\begin{aligned}
 D_\Phi(x^*):=\{d\in\R^n:\exists\,t>0,x^*+td\ne0,\Phi(x^*+td)\le\Phi(x^*)\}.
\end{aligned}
\end{equation}
Both cones contain zero and are invariant under positive scaling.
The radial cone requires $x^*+td\in C$ for all sufficiently small
$t>0$; the global descent cone requires only one such step.
The latter need not be convex or closed. Intersecting either cone
with $\ker(A)$ restricts it to feasible directions.

Following \cite[Eq.~(4.5)]{ZYP}, define the auxiliary function
\begin{equation*}
  \Psi_{x^*}(x):=\norm{x}_1-\Phi(x^*)\norm{x}_2,
 \quad x\in\R^n.
\end{equation*}
For every $x\ne0$,
\begin{equation}\label{Eq_gap}
 \Psi_{x^*}(x)=\norm{x}_2\bigl(\Phi(x)-\Phi(x^*)\bigr).
\end{equation}
For $x\ne0$, $x\in C$ if and only if $\Psi_{x^*}(x)\le0$, so objective
comparisons can be expressed without division by a norm.
The radial cone is determined exactly by the directional derivative.

\begin{lemma}\label{Prop_localcones}
For every $x^*\ne0$, the radial cone and the tangent cone satisfy
\begin{equation}\label{Eq_localcones}
 \mathcal R_C(x^*)=\mathcal T_C(x^*)
 =\{d\in\R^n:\Phi'(x^*;d)\le0\}.
\end{equation}
These cones are closed and convex, and
$\mathcal R_C(x^*)\subseteq D_\Phi(x^*)$.
\end{lemma}
\begin{proof}
The  inclusions $\mathcal R_C(x^*)\subseteq
\mathcal T_C(x^*)$ and $\mathcal R_C(x^*)\subseteq D_\Phi(x^*)$ are trivial by their definitions.

Fix any $d\in\mathcal T_C(x^*)$. There exist sequences  $t_k\downarrow0$ and $d_k\to d$
such that $x^*+t_kd_k\in C$. Equations~\eqref{Eq_derivative}
and~\eqref{Eq_expansion} give
\[
 0\ge\frac{\Phi(x^*+t_kd_k)-\Phi(x^*)}{t_k}
 =\Phi'(x^*;d_k)+O(t_k\norm{d_k}_2^2).
\]
Passing to the limit gives $\Phi'(x^*;d)\le0$.
Thus $\mathcal T_C(x^*)\subseteq\{d\in\R^n:\Phi'(x^*;d)\le0\}$.

Conversely, let $\Phi'(x^*;d)\le0$. If the inequality is strict,
then $\Phi(x^*+td)<\Phi(x^*)$ for all sufficiently small $t>0$.
If $\Phi'(x^*;d)=0$, Lemma~\ref{Lem_ray} gives the same strict
inequality when $d$ is not parallel to $x^*$, and equality otherwise.
In either case, $d\in\mathcal R_C(x^*)$. Together with the preceding
inclusions, this proves \eqref{Eq_localcones}.

Continuity and convexity of $\Phi'(x^*;\cdot)$ imply closedness and
convexity of these cones. 
\end{proof}

The next lemma gives an explicit description of the global descent cone.
For directions not parallel to $x^*$, membership is determined by an
asymptotic comparison and comparisons at finitely many positive steps
where a component becomes zero.
\begin{lemma}\label{Prop_descent}
Let $x^*\ne0$ and $S=\supp(x^*)$. For each $i\in S$, define
\[
 E_i:=\left\{d\in\R^n:\ x_i^*d_i<0,\quad
 \norm{d_i x^*-x_i^*d}_1
 \le\Phi(x^*)\norm{d_i x^*-x_i^*d}_2\right\}.
\]
Then
\begin{equation}\label{Eq_cone}
 D_\Phi(x^*)=
 \operatorname{span}\{x^*\}
 \ \cup\ \{d\ne0:\Phi(d)<\Phi(x^*)\}
 \ \cup\ \bigcup_{i\in S}E_i.
\end{equation}
In particular, for $d\notin\operatorname{span}\{x^*\}$,
$d\in D_\Phi(x^*)$ if and only if either
$\Phi(d)<\Phi(x^*)$ or
\[
 \Phi\left(x^*-\frac{x_i^*}{d_i}d\right)\le\Phi(x^*)
 \quad\text{for some }i\in S\text{ with }x_i^*d_i<0.
\]
Thus, besides the comparison involving $\Phi(d)$, at most $|S|$
additional comparisons are required.
\end{lemma}
\begin{proof}
Set $\omega=\Phi(x^*)$. For each $d\in\R^n$, define
\[
 F_d(t):=\Psi_{x^*}(x^*+td)
 =\norm{x^*+td}_1-\omega\norm{x^*+td}_2,
 \qquad t\ge0.
\]
This function is continuous and satisfies $F_d(0)=0$.
For every $t>0$ such that $x^*+td\ne0$, identity~\eqref{Eq_gap}
implies that $F_d(t)\le0$ if and only if
$\Phi(x^*+td)\le\omega$.

For the inclusion of the right-hand side of \eqref{Eq_cone}
in $D_\Phi(x^*)$, consider each set in the union.
If $d\in\operatorname{span}\{x^*\}$, scale invariance gives
$\Phi(x^*+td)=\omega$ for all sufficiently small $t>0$.
If $d\ne0$ and $\Phi(d)<\omega$, scale invariance and continuity
of $\Phi$ at $d$ give
\[
\begin{aligned}
 \lim_{t\to\infty}\Phi(x^*+td)
=\lim_{t\to\infty}\Phi\left(\frac{x^*+td}{t}\right)=\lim_{t\to\infty}\Phi\left(d+\frac{x^*}{t}\right)
 =\Phi(d)<\omega,
\end{aligned}
\]
so a sufficiently large positive step gives $d\in D_\Phi(x^*)$.
Finally, let $d\in E_i$ for some $i\in S$ and set
$t_i=-x_i^*/d_i>0$. If $x^*+t_i d=0$, then
$d\in\operatorname{span}\{x^*\}$, which has already been considered.
Otherwise, the defining inequality for $E_i$ gives
\[
 |d_i|F_d(t_i)
 =\norm{d_i x^*-x_i^*d}_1
  -\omega\norm{d_i x^*-x_i^*d}_2\le0.
\]
Hence $d\in D_\Phi(x^*)$ in this case as well.

For the reverse inclusion, consider a direction outside the
union in \eqref{Eq_cone}; that is, let $d\in\R^n$ satisfy
\[
 d\notin\operatorname{span}\{x^*\},\qquad
 \norm d_1-\omega\norm d_2\ge0,
 \qquad d\notin E_i\quad\text{for all }i\in S.
\]
In particular, $x^*+td\ne0$ for every $t\ge0$.
Let $0<t_1<\cdots<t_N$ be the distinct positive steps at which
$x_i^*+td_i=0$ for some $i\in S$, and put $t_0=0$.
The list is finite and may be empty, in which case $N=0$.
For each $j\in\{1,\ldots,N\}$, choose an index $i\in S$ such that
$t_j=-x_i^*/d_i>0$. Since $x_i^*d_i<0$ and $d\notin E_i$,
\[
 |d_i|F_d(t_j)
 =\norm{d_i x^*-x_i^*d}_1
  -\omega\norm{d_i x^*-x_i^*d}_2>0.
\]
Thus $F_d(t_j)>0$ at every such step.

Suppose first that $N\ge1$, and fix $j\in\{0,\ldots,N-1\}$.
By construction, $x_i^*+td_i$ is nonzero throughout
$(t_j,t_{j+1})$ for each $i\in S$ and therefore has a constant
sign there, denoted by $\sigma_i\in\{-1,1\}$.
The identities $x^*_{S^c}=0$ and $|td_i|=t|d_i|$ for $t\ge0$
give
\[
 \norm{x^*+td}_1
 =\sum_{i\in S}\sigma_i(x_i^*+td_i)
   +t\norm{d_{S^c}}_1,
 \qquad t\in[t_j,t_{j+1}].
\]
Here continuity extends the identity to both endpoints, even if
a component becomes zero there. The right-hand side is affine in
$t$. Since $t\mapsto-\omega\norm{x^*+td}_2$ is concave by
convexity of the Euclidean norm and $\omega>0$, $F_d$ is concave
on $[t_j,t_{j+1}]$ and hence bounded below by the chord joining
its endpoint values. These values are positive except at $t_0=0$,
where $F_d(t_0)=0$. It follows that $F_d(t)>0$ for every
$t\in(0,t_N]$.

It remains to consider $(t_N,\infty)$, with $t_N=t_0=0$ when
$N=0$. All positive zeros of the components have been passed,
so $x_i^*+td_i$ has the sign of $d_i$ whenever $d_i\ne0$.
Thus $F_d$ is differentiable on this interval, with
\[
 F_d'(t)=\norm d_1-
 \omega\frac{\ip{x^*+td}{d}}{\norm{x^*+td}_2}
 >\norm d_1-\omega\norm d_2\ge0.
\]
The strict inequality follows from the Cauchy--Schwarz inequality:
$d\notin\operatorname{span}\{x^*\}$ implies that $x^*+td$ is
not parallel to $d$. Hence $F_d$ is strictly increasing on
$(t_N,\infty)$ and continuous at $t_N$.
If $N\ge1$, the preceding concavity argument gives $F_d(t)>0$
on $(0,t_N]$, while strict increase from $F_d(t_N)>0$ gives
$F_d(t)>0$ on $(t_N,\infty)$. These two intervals cover
$(0,\infty)$. If $N=0$, the final interval is already
$(0,\infty)$, and strict increase from $F_d(0)=0$ gives the
same conclusion. Thus, in either case, $F_d(t)>0$ for every
$t>0$. Since $x^*+td\ne0$, identity~\eqref{Eq_gap} yields
\[
 \Phi(x^*+td)>\Phi(x^*)\qquad\text{for every }t>0.
\]
No positive step therefore satisfies the defining inequality
of $D_\Phi(x^*)$, so $d\notin D_\Phi(x^*)$. This completes the
proof of the reverse inclusion and establishes \eqref{Eq_cone}.

Finally, membership in $E_i$ admits an equivalent objective-value
comparison. Fix $d\notin\operatorname{span}\{x^*\}$ and an index
$i\in S$ such that $x_i^*d_i<0$. The positive step
$t_i=-x_i^*/d_i$ makes the $i$th component of $x^*+t_i d$ zero.
The vector $x^*+t_i d$ itself is nonzero, since otherwise $d$
would be parallel to $x^*$. Since
$d_i x^*-x_i^*d=d_i(x^*+t_i d)$ and $d_i\ne0$,
absolute homogeneity of the norms shows that the defining inequality
for $E_i$ is equivalent to $\Phi(x^*+t_i d)\le\omega$.
Since $d\notin\operatorname{span}\{x^*\}$, formula~\eqref{Eq_cone}
shows that $d\in D_\Phi(x^*)$ precisely when either
$\Phi(d)<\Phi(x^*)$ or $d\in E_i$ for some $i\in S$.
The equivalence just established gives the stated membership test.
Each index $i\in S$ contributes at most one such step, so at most
$|S|$ additional comparisons are required.
\end{proof}

The following outer estimates of the global descent cone yield
sufficient recovery conditions.
\begin{lemma}\label{Prop_outer}
Let $x^*\ne0$ and $S=\supp(x^*)$. The global descent cone satisfies
\begin{equation}\label{Eq_outer}
\begin{aligned}
 D_\Phi(x^*)
 &\subseteq\left\{d\in\R^n:\ \norm{d_{S^c}}_1
 \le\sqrt{|S|}\,\operatorname{dist}
       (d,\operatorname{span}\{x^*\})\right\}\\
 &\subseteq\left\{d\in\R^n:\ \norm{d_{S^c}}_1
 \le\sqrt{|S|}\norm d_2\right\},
\end{aligned}
\end{equation}
where
\[
 \operatorname{dist}(d,\operatorname{span}\{x^*\})
 =\left(\norm d_2^2-
       \frac{\ip{x^*}{d}^2}{\norm{x^*}_2^2}\right)^{1/2}.
\]
\end{lemma}
\begin{proof}
Fix any $d\!\in\! D_\Phi(x^*)$. Choose $\tau\!>\!0$ such that
$y\!=\!x^*+\tau d\!\ne\!0$ and $\Phi(y)\!\le\!\Phi(x^*)$, and set
\[
 u=\frac{x^*}{\norm{x^*}_2},\qquad
 v=\frac{y}{\norm y_2},\qquad \alpha=\ip uv.
\]
Then $\norm u_2=\norm v_2=1$, $|\alpha|\le1$, and
$u_{S^c}=0$. Since
\[
 \norm v_1=\frac{\norm y_1}{\norm y_2}
 =\Phi(y)\le\Phi(x^*)
 =\norm u_1,
\]
it follows that
\begin{equation}\label{Eq_outer_support_gap}
 \norm{v_{S^c}}_1
 =\norm v_1-\norm{v_S}_1
 \le\norm{u_S}_1-\norm{v_S}_1.
\end{equation}

Using $u_S=(u-\alpha v)_S+\alpha v_S$, the triangle inequality gives
\[
 \norm{u_S}_1
 \le\norm{(u-\alpha v)_S}_1+|\alpha|\norm{v_S}_1.
\]
Substituting this bound into \eqref{Eq_outer_support_gap}, using
$|\alpha|\le1$, and applying the Cauchy--Schwarz inequality yields
\begin{equation}\label{Eq_outer_normalized_bound}
 \norm{v_{S^c}}_1
 \le\norm{(u-\alpha v)_S}_1
 \le\sqrt{|S|}\norm{u-\alpha v}_2.
\end{equation}
Here $\norm{(u-\alpha v)_S}_2\le\norm{u-\alpha v}_2$.
Also,
\[
 \norm{u-\alpha v}_2^2
 =1-\alpha^2
 =\norm{v-\alpha u}_2^2.
\]
Since $\alpha u$ is the orthogonal projection of $v$ onto
$\operatorname{span}\{x^*\}$, it follows that
\begin{equation}\label{Eq_outer_projection_distance}
 \norm y_2\norm{u-\alpha v}_2
 =\operatorname{dist}(y,\operatorname{span}\{x^*\})
 =\tau\operatorname{dist}(d,\operatorname{span}\{x^*\}).
\end{equation}
The last equality follows from $y=x^*+\tau d$ and
$x^*\in\operatorname{span}\{x^*\}$.
Also, $y_{S^c}=\tau d_{S^c}$ gives
$\norm y_2\norm{v_{S^c}}_1=\tau\norm{d_{S^c}}_1$.
Multiplying \eqref{Eq_outer_normalized_bound} by $\norm y_2/\tau$
and using \eqref{Eq_outer_projection_distance} therefore yields
\[
 \norm{d_{S^c}}_1
 \le\sqrt{|S|}\operatorname{dist}(d,\operatorname{span}\{x^*\}).
\]
This proves the first inclusion in \eqref{Eq_outer}; the second
follows from
$\operatorname{dist}(d,\operatorname{span}\{x^*\})\le\norm d_2$.
\end{proof}
Both outer cones in \eqref{Eq_outer} are closed and symmetric and
contain $\cl D_\Phi(x^*)$. The first depends on the direction of
$x^*$; the second depends only on its support and will give the
uniform condition in Theorem~\ref{Th_uniform_global}.

\section{Exact Recovery and Stability at a Fixed Signal}\label{Sec_pointwise}
Fix a signal $x^*\in\Delta$. The directional and tangent-cone
formulas of Section~\ref{Sec_geometry}, together with the curvature
of zero-slope feasible rays, yield equivalent criteria for local
exact recovery, first-order growth, and tilt stability. The global
descent cone then characterizes unique global recovery, and its
explicit description gives a test using coordinate-vanishing steps
and an asymptotic comparison. The resulting pointwise criteria are
the starting point for the uniform estimates in
Section~\ref{Sec_uniform}.

\subsection{Local Exact Recovery and Stability}
A signal $x^*\in\Delta$ is locally exactly recovered if it uniquely
minimizes $\Phi$ over the feasible points in some neighborhood of
$x^*$. Equivalently, $x^*$ is a \emph{strict local minimizer} of
$\Phi$ on $\Delta$; that is, there exists $\rho>0$ such that
$\Phi(x)>\Phi(x^*)$ for every
$x\in\Delta$ with $0<\norm{x-x^*}_2<\rho$. A quantitative version requires the objective gap to bound the
distance to $x^*$ from below by a positive linear term.

\begin{definition}
\label{Def_sharp}
Let $x^*\in\Delta$. The point $x^*$ is a \emph{sharp local minimizer}
of $\Phi$ on $\Delta$ if there exist constants $\kappa>0$ and
$\rho>0$ such that
\begin{equation}\label{Eq_growth}
 \Phi(x)\ge\Phi(x^*)+\kappa\norm{x-x^*}_2
 \qquad\text{for all }x\in\Delta\text{ with }\norm{x-x^*}_2\le\rho.
\end{equation}
In this case, $\Phi$ is said to have \emph{first-order growth} at
$x^*$ relative to $\Delta$, and $\kappa$ is called a local growth
modulus.
\end{definition}

Sharp local minima are strict. The converse fails in general but
holds for this model, as shown in Proposition~\ref{Th_local}.

To study recovery through objective levels, as in
\cite[Theorem 3.1]{LPB}, define the \emph{equality-level mapping}
\begin{equation}\label{Eq_level_mapping}
 \mathcal E(\tau):=\{x\in\Delta:\Phi(x)=\Phi(x^*)+\tau\},
 \qquad \tau\in\R.
\end{equation}
Thus $\mathcal E(\tau)$ consists of feasible points with objective
gap $\tau$. First-order growth will be related to isolated
calmness of this mapping, defined below.

\begin{definition}\label{Def_calmness}
Let $P$ be a finite-dimensional Euclidean parameter space,
$M:P\rightrightarrows\R^n$ a set-valued mapping, and
$\overline{x}\in M(\bar p)$. The mapping $M$ is \emph{isolated calm at
$\bar p$ for $\overline{x}$} if there exist a neighborhood $U$ of $\bar p$,
a neighborhood $V$ of $\overline{x}$, and a constant $\kappa>0$ such that
\[
 M(p)\cap V\subseteq\{\overline{x}\}+\kappa\norm{p-\bar p}_2\mathbb B
 \qquad(p\in U).
\]
In particular, $M(\bar p)\cap V=\{\overline{x}\}$. Perturbed values
$M(p)\cap V$ may be empty or multivalued.
\end{definition}

The following characterization is due to King and Rockafellar
\cite[Proposition 2.1]{KR}; see also \cite[Lemma 2.1]{LPB}.

\begin{lemma}\label{Lem_calmness}
Let $P$ be a finite-dimensional Euclidean space, let
$M:P\rightrightarrows\R^n$, and let $(\bar p,\overline{x})\in\gph M$.
Then $M$ is isolated calm at $\bar p$ for $\overline{x}$ if and only if
\[
 DM(\bar p\mid \overline{x})(0)=\{0\}.
\]
\end{lemma}

The next theorem relates first-order growth to directional and
geometric conditions and to isolated calmness of $\mathcal E$.

\begin{theorem}\label{Th_growth}
Let $x^*\in\Delta$, and let $C$ and $\mathcal E$ be defined by
\eqref{Eq_sublevel} and \eqref{Eq_level_mapping}, respectively.
Then the following statements are equivalent:
\begin{itemize}
\item[(i)] $x^*$ is a sharp local minimizer of $\Phi$ on $\Delta$.
\item[(ii)] $\ker(A)\cap
\{d\in\R^n:\Phi'(x^*;d)\le0\}=\{0\}$.
\item[(iii)] $\norm{d_{S^c}}_1>
\left|\ip{\sgn(x^*_S)}{d_S}
-\frac{\norm{x^*}_1}{\norm{x^*}_2^2}\ip{x^*}{d}\right|$ for every
$d\in\ker(A)\setminus\{0\}$.
\item[(iv)] $\ker(A)\cap \mathcal{T}_C(x^*)=\{0\}$.
\item[(v)] $x^*$ is a local minimizer of $\Phi$ on $\Delta$, and
$\mathcal E$ is isolated calm at $0$ for $x^*$
(in the sense of Definition~\ref{Def_calmness}).
\end{itemize}
If $\ker(A)\ne\{0\}$, these conditions are also equivalent to
\begin{equation}\label{Eq_alpha}
 \alpha:=\min_{d\in\ker(A),\ \norm d_2=1}\Phi'(x^*;d)>0.
\end{equation}
Any $0<\kappa<\alpha$ is a local growth modulus for a sufficiently
small neighborhood of $x^*$.
\end{theorem}
\begin{proof}
If $\ker(A)=\{0\}$, then $\Delta=\{x^*\}$ and all five
conditions hold. Suppose henceforth that $\ker(A)\ne\{0\}$.
The equivalence follows from the cycle of implications
\[
 \text{(i)}\ \Longrightarrow\ \text{(ii)}\ \Longrightarrow\
 \text{(iii)}\ \Longrightarrow\ \text{(iv)}\ \Longrightarrow\
 \text{(v)}\ \Longrightarrow\ \text{(i)}.
\]

\emph{(i) $\Longrightarrow$ (ii).}
Let $\kappa>0$ be a local growth modulus. For any $d\in\ker(A)$,
applying \eqref{Eq_growth} to $x^*+td$ for sufficiently small
$t>0$, dividing by $t$, and letting $t\downarrow0$ gives
\[
 \Phi'(x^*;d)\ge\kappa\norm d_2.
\]
Thus the directional derivative is positive for every nonzero
$d\in\ker(A)$, proving (ii).

\emph{(ii) $\Longrightarrow$ (iii).}
For any nonzero $d\in\ker(A)$, condition (ii) and
\eqref{Eq_derivative} give
\[
 \ip{\sgn(x^*_S)}{d_S}
 -\frac{\norm{x^*}_1}{\norm{x^*}_2^2}\ip{x^*}{d}
 +\norm{d_{S^c}}_1>0.
\]
Applying this inequality to $d$ and $-d$ bounds the linear expression
strictly between $-\norm{d_{S^c}}_1$ and $\norm{d_{S^c}}_1$,
which is precisely (iii).

\emph{(iii) $\Longrightarrow$ (iv).}
Condition (iii) and \eqref{Eq_derivative} give
$\Phi'(x^*;d)>0$ for every nonzero $d\in\ker(A)$.
Lemma~\ref{Prop_localcones} identifies
$\mathcal T_C(x^*)$ with $\{d:\Phi'(x^*;d)\le0\}$,
so its intersection with $\ker(A)$ is $\{0\}$.

\emph{(iv) $\Longrightarrow$ (v).}
By Lemma~\ref{Prop_localcones}, (iv) implies strict positivity of
$\Phi'(x^*;d)$ on the unit sphere of $\ker(A)$. Continuity and
compactness give $\alpha>0$ in \eqref{Eq_alpha}.
For $d\in\ker(A)$ tending to zero, positive homogeneity and
\eqref{Eq_expansion} yield
\[
 \Phi(x^*+d)-\Phi(x^*)
 \ge\alpha\norm d_2+O(\norm d_2^2).
\]
For any $0<\kappa<\alpha$, there is a radius $\rho>0$ such that,
for every $x\in\Delta$ with $\norm{x-x^*}_2\le\rho$,
\[
 \Phi(x)-\Phi(x^*)\ge\kappa\norm{x-x^*}_2.
\]
In particular, $x^*$ is a strict local minimizer. For every
$x\in\mathcal E(\tau)$ in this neighborhood,
\[
 \norm{x-x^*}_2\le\kappa^{-1}\tau
 \le\kappa^{-1}|\tau|.
\]
Hence $\mathcal E$ is isolated calm at $0$ for $x^*$, proving (v).
This argument also establishes the asserted growth moduli.

\emph{(v) $\Longrightarrow$ (i).}
By Definition~\ref{Def_calmness}, there exist constants
$\varepsilon>0$, $\delta>0$, and $L>0$ such that, for every
$\tau\in\R$ with $|\tau|<\varepsilon$,
\[
 \mathcal E(\tau)\cap(x^*+\delta\mathbb B)
 \subseteq x^*+L|\tau|\mathbb B.
\]
Continuity of $\Phi$ at $x^*$ and local minimality yield a radius
$0<\rho\le\delta$ such that, for every $x\in\Delta$ with
$\norm{x-x^*}_2\le\rho$,
\[
 0\le\Phi(x)-\Phi(x^*)<\varepsilon.
\]
For any such $x$, set $\tau=\Phi(x)-\Phi(x^*)$. Then
$x\in\mathcal E(\tau)\cap(x^*+\delta\mathbb B)$, and the
inclusion above gives
\[
 \norm{x-x^*}_2\le L|\tau|
 =L\bigl(\Phi(x)-\Phi(x^*)\bigr).
\]
Rearranging proves \eqref{Eq_growth} with modulus $1/L$ and radius
$\rho$, establishing (i).

Finally, by continuity and positive homogeneity of
$\Phi'(x^*;\cdot)$, condition (ii) is equivalent to
$\alpha>0$ in \eqref{Eq_alpha}.
\end{proof}

\begin{remark}\label{Rem_cones}
The local minimality assumption in Theorem~\ref{Th_growth}(v)
is necessary: isolated calmness of $\mathcal E$ alone does not imply
local exact recovery. In fact, for every $u\in\R$,
\[
 D\mathcal E(0\mid x^*)(u)
 =\{d\in\ker(A):\Phi'(x^*;d)=u\}.
\]
Indeed, if $d\in D\mathcal E(0\mid x^*)(u)$, there exist
$t_k\downarrow0$ and $(u_k,d_k)\to(u,d)$ such that
$(t_ku_k,x^*+t_kd_k)\in\gph\mathcal E$.
Feasibility gives $Ad_k=0$ and hence $Ad=0$.
The definition of $\mathcal E$ also gives
\[
 \Phi(x^*+t_kd_k)=\Phi(x^*)+t_ku_k.
\]
Rearranging, dividing by $t_k>0$, and applying
\eqref{Eq_expansion} gives
\[
 u_k=\frac{\Phi(x^*+t_kd_k)-\Phi(x^*)}{t_k}
 =\Phi'(x^*;d_k)+O(t_k\norm{d_k}_2^2)\longrightarrow\Phi'(x^*;d).
\]
Conversely, if $d\in\ker(A)$ and $\Phi'(x^*;d)=u$, set
$x_t=x^*+td$ and $\tau_t=\Phi(x_t)-\Phi(x^*)$ for sufficiently
small $t>0$. Then $(\tau_t,x_t)\in\gph\mathcal E$ and
\[
 \left(\frac{\tau_t}{t},\frac{x_t-x^*}{t}\right)
 \longrightarrow(u,d),
\]
which proves the reverse inclusion. Lemma~\ref{Lem_calmness}
therefore gives
\[
 \mathcal E\text{ is isolated calm at }0\text{ for }x^*
 \quad\Longleftrightarrow\quad
 \ker(A)\cap\{d\in\mathbb R^n:\Phi'(x^*;d)=0\}=\{0\}.
\]
This condition excludes zero slopes but allows negative ones.
Local minimality rules out negative slopes, and the two conditions
then imply first-order growth by Theorem~\ref{Th_growth}.
In \cite[Theorem 3.1]{LPB}, membership in the global optimal level
already ensures minimality; Theorem~\ref{Th_growth}(v) includes
local minimality explicitly.

For example, let $A=(1\ \ 0)$, $b=2$, and $x^*=(2,1)$.
For any feasible direction $d=(0,v)$, \eqref{Eq_derivative} gives
\[
 \Phi'(x^*;d)=\frac{2v}{5\sqrt5}.
\]
The derivative vanishes only when $d=0$, so $\mathcal E$ is
isolated calm at $0$ for $x^*$. However, it is negative for $v<0$,
and hence $x^*$ is not a local minimizer.
\end{remark}

Under linear objective perturbations, tilt stability requires
the localized minimizer to be unique and to depend Lipschitz
continuously on the perturbation.

\begin{definition}\label{Def_tilt}
A point $x^*\in\Delta$ is a \emph{tilt-stable local minimizer} of $f$
if there exist a radius $\rho>0$ and a neighborhood $U$ of zero
in $\R^n$ such that the localized minimizer mapping
\begin{equation}\label{Eq_tilt}
 T_\rho(v)=\argmin_{\norm{x-x^*}_2\le\rho}
          \{f(x)-\ip v{x-x^*}\}
\end{equation}
has exactly one element for every $v\in U$, depends Lipschitz
continuously on $v$, and satisfies $T_\rho(0)=\{x^*\}$.
Each singleton $T_\rho(v)$ is identified with its unique element.
This is the notion introduced by Poliquin and Rockafellar~\cite{PR}.
\end{definition}

\begin{proposition}\label{Th_local}
Let $x^*\in\Delta$. Then the following
statements are equivalent:
\begin{itemize}
\item[(i)] $x^*$ is a local minimizer of $\Phi$ on $\Delta$.
\item[(ii)] $x^*$ is a strict local minimizer of $\Phi$ on $\Delta$.
\item[(iii)] $x^*$ is a sharp local minimizer of $\Phi$ on $\Delta$.
\item[(iv)] $x^*$ is a tilt-stable local minimizer of $f$.
\end{itemize}
These properties are also equivalent to each condition in
Theorem~\ref{Th_growth}. Moreover, if $\kappa>0$ and $\rho>0$
satisfy \eqref{Eq_growth}, then the mapping in \eqref{Eq_tilt}
satisfies
\begin{equation}\label{Eq_fixed}
 T_\rho(v)=\{x^*\}\qquad\text{for every }\norm v_2<\kappa.
\end{equation}
\end{proposition}
\begin{proof}
The implications (iii) $\Longrightarrow$ (ii) $\Longrightarrow$ (i)
follow directly from the definitions. It remains to prove
(i) $\Longrightarrow$ (iii) $\Longrightarrow$ (iv)
$\Longrightarrow$ (i).

\emph{(i) $\Longrightarrow$ (iii).}
Suppose that $x^*$ is a local minimizer. For any $d\in\ker(A)$,
the points $x^*+td$ are feasible, so local minimality gives
$\Phi'(x^*;d)\ge0$.

Suppose that $\Phi'(x^*;d)=0$ for some nonzero $d\in\ker(A)$.
Remark~\ref{Rem_feasible_curvature} gives $\Phi''(x^*;d)<0$, so
\[
 \Phi(x^*+td)-\Phi(x^*)
 =\frac12\Phi''(x^*;d)t^2+o(t^2)<0
\]
for all sufficiently small $t>0$, contradicting local minimality.
Hence $\Phi'(x^*;d)>0$ for every nonzero $d\in\ker(A)$.
Theorem~\ref{Th_growth} therefore gives first-order growth,
which is precisely (iii).

\emph{(iii) $\Longrightarrow$ (iv).}
Let $\kappa>0$ and $\rho>0$ satisfy \eqref{Eq_growth}.
For any $v\in\R^n$ with $\norm v_2<\kappa$ and any $x\in\Delta$
with $\norm{x-x^*}_2\le\rho$, first-order growth and
the Cauchy--Schwarz inequality give
\[
\begin{aligned}
 f(x)-\ip{v}{x-x^*}-f(x^*)
 &\ge\kappa\norm{x-x^*}_2-\ip{v}{x-x^*}\\
 &\ge(\kappa-\norm v_2)\norm{x-x^*}_2.
\end{aligned}
\]
For $x\ne x^*$ the right-hand side is positive, and infeasible
points have infinite objective. This proves \eqref{Eq_fixed}.
Since $T_\rho$ is constant for $\norm v_2<\kappa$, it is Lipschitz
continuous there, proving (iv).

\emph{(iv) $\Longrightarrow$ (i).}
Tilt stability gives $T_\rho(0)=\{x^*\}$ for some $\rho>0$,
so $x^*$ is a local minimizer of $\Phi$ on $\Delta$.
The remaining equivalences follow from Theorem~\ref{Th_growth}.
\end{proof}

\begin{remark}\label{Rem_local_LP}
The local recovery criterion admits a formulation in terms of linear
programming feasibility. Fix $x^*\in\Delta$ and set
$\gamma:=\norm{x^*}_1/\norm{x^*}_2^2$.
By Theorem~\ref{Th_growth} and Proposition~\ref{Th_local}, $x^*$
is locally exactly recovered if and only if, for every
$i\in\{1,\ldots,n\}$ and $\sigma\in\{-1,1\}$, the following system
in $(d,w)\in\R^n\times\R^{|S^c|}$ is infeasible:
\[
\begin{cases}
Ad=0,\qquad d_i=\sigma,\\
\displaystyle
\ip{\sgn(x^*_S)}{d_S}
+\sum_{j\in S^c}w_j-\gamma\ip{x^*}{d}\le0,\\
-w_j\le d_j\le w_j,\qquad j\in S^c.
\end{cases}
\]
Here $w=(w_j)_{j\in S^c}$.
Indeed, the inequalities involving $w$ give an exact representation
of the condition $\Phi'(x^*;d)\le0$ after eliminating $w$.
Since this condition and $Ad=0$ define a cone, any nonzero vector
in that cone can be positively scaled so that one coordinate equals
$1$ or $-1$. Thus at most $2n$ linear programming feasibility
problems suffice to decide local exact recovery at a given signal.
If $A$ and $x^*$ have rational entries, this decision can be made
in polynomial time in their total binary encoding length.
\end{remark}

\begin{remark}\label{Rem_local}
The perturbation estimate in Proposition~\ref{Th_local} shows that
first-order growth implies tilt stability for general objectives.
The converse fails even for a smooth convex objective with a unique
minimizer and nonzero data. Consider
\[
 F(x_1,x_2)=x_2^2,\qquad A=(1\ \ 0),\qquad x^*=(1,0).
\]
Then $Ax^*=1\ne0$ and the feasible set is
$\Delta=\{(1,t):t\in\R\}$. The point $x^*$ uniquely minimizes $F$
on $\Delta$, but it is not sharp, since, as $t\to0$ through nonzero values,
\[
 \frac{F(1,t)-F(1,0)}{\norm{(1,t)-(1,0)}_2}
 =|t|\longrightarrow0.
\]
Nevertheless, the localized tilted minimizer is
\[
 T_\rho(v)=\{(1,v_2/2)\}
\]
for every $v=(v_1,v_2)\in\R^2$ with $\norm v_2<2\rho$.
This mapping is single-valued and Lipschitz continuous, so $x^*$ is
tilt stable. Thus tilt stability allows the minimizer to move with
the perturbation, whereas first-order growth keeps it fixed for
sufficiently small perturbations.

For the ratio objective, zero-slope rays have negative curvature
along every nonzero feasible direction because $Ax^*\ne0$.
In the example above, the curvature is positive instead.

\end{remark}

\begin{corollary}\label{Cor_support}
Let $x^*$ be a local minimizer of $\Phi$ on $\Delta$. Then
\[
 \ker(A)\cap\{d\in\R^n:d_{S^c}=0\}=\{0\}.
\]
Equivalently, $A_S$ has full column rank. In particular,
$|S|\le\operatorname{rank}(A)$.
Moreover, the set of local minimizers of $\Phi$ on $\Delta$ is finite.
\end{corollary}
\begin{proof}
By Proposition~\ref{Th_local}, $x^*$ satisfies
Theorem~\ref{Th_growth}(iii). If a nonzero $d\in\ker(A)$ were
supported on $S$, that strict inequality would have zero on its
left-hand side and a nonnegative quantity on its right-hand side,
a contradiction. Thus the displayed intersection is trivial.
Since $Ad=A_Sd_S$ for every $d$ supported on $S$, this is precisely
injectivity of $A_S$, or equivalently full column rank. Hence
$|S|=\operatorname{rank}(A_S)\le\operatorname{rank}(A)$.

For each index set $T$, if $A_T$ has full column rank, then
$A_Tx_T=b$ has at most one solution. Thus at most one local
minimizer has support $T$. Since there are
only finitely many subsets of $\{1,\ldots,n\}$, the set of local
minimizers is finite.
\end{proof}

\begin{remark} (i)
The support-rank conclusion in Corollary~\ref{Cor_support} is the
necessary condition established by Tao and Zhang~\cite[Theorem 5]{TZ}.
Their Corollary~2 also establishes finiteness of the local minimizer
set. Here it follows from Theorem~\ref{Th_growth}, which also gives
positive directional derivatives along all nonzero feasible directions,
first-order growth, and tilt stability.

(ii) The rank condition alone does
not imply local minimality. For example, let
\[
 A=\begin{pmatrix}1&0&10\\0&1&0\end{pmatrix},\qquad
 x^*=(2,1,0),\qquad d=(10,0,-1),
\]
$A_S=I_2$ and $Ad=0$, whereas
$\Phi'(x^*;d)=-1/\sqrt5<0$ by \eqref{Eq_derivative}.

\end{remark}

\begin{proposition}
\label{Th_critical}
Let $x^*\in\Delta$ be a critical point of $\Phi$ on $\Delta$.
Then exactly one of the following alternatives holds:
\begin{itemize}
\item[(i)] $x^*$ is a sharp local minimizer of $\Phi$ on $\Delta$.
\item[(ii)] There exists $d\in\ker(A)\setminus\{0\}$ such that
\[
 \Phi'(x^*;d)=0,\qquad
 \Phi''(x^*;d)<0.
\]
In this case, $\Phi(x^*+td)<\Phi(x^*)$ for every sufficiently small
$t>0$, so $x^*$ is not a local minimizer.
\end{itemize}
\end{proposition}
\begin{proof}
By Proposition~\ref{Prop_stationary}, criticality gives
$\Phi'(x^*;d)\ge0$ for every $d\in\ker(A)$.
If the inequality is strict for every nonzero $d\in\ker(A)$,
Theorem~\ref{Th_growth} yields (i).
Otherwise, there exists a nonzero $d\in\ker(A)$ with
$\Phi'(x^*;d)=0$. Remark~\ref{Rem_feasible_curvature} gives
$\Phi''(x^*;d)<0$, so
\[
 \Phi(x^*+td)-\Phi(x^*)
 =\frac12\Phi''(x^*;d)t^2+o(t^2)<0
\]
for all sufficiently small $t>0$. Thus (ii) holds and $x^*$ is
not a local minimizer, so (i) cannot hold.
\end{proof}

Under the full-support assumption, the second alternative in
Proposition~\ref{Th_critical} can be strengthened to strict local maximality.

\begin{corollary}
\label{Cor_fullsupport}
Suppose $\ker(A)\ne\{0\}$. Every critical point $x^*\in\Delta$
with $\supp(x^*)=\{1,\ldots,n\}$ is a strict local maximizer of
$\Phi$ on $\Delta$.
\end{corollary}
\begin{proof}
Since all coordinates of $x^*$ are nonzero, $\Phi$ is smooth near
$x^*$. Its derivative is linear in $d$. Stationarity applied to
$d$ and $-d$ therefore gives $\Phi'(x^*;d)=0$ for every
$d\in\ker(A)$. Combining
\eqref{Eq_no_parallel_feasible} with \eqref{Eq_curvature} gives
\[
 \ip{d}{\nabla^2\Phi(x^*)d}<0
 \qquad(d\in\ker(A)\setminus\{0\}).
\]
Thus the Hessian is negative definite on the feasible direction space.
Compactness of its unit sphere and the second-order Taylor expansion
imply strict local maximality on $\Delta$.
\end{proof}

\begin{example}\label{Ex_stationary}
Let $A=(1\ \ 1)$, $b=2$, and $x^*=(1,1)$. For $d=(1,-1)$,
$\Phi'(x^*;d)=\Phi'(x^*;-d)=0$, but for small nonzero $t$,
\[
 \Phi(1+t,1-t)=\frac2{\sqrt{2+2t^2}}<\sqrt2=\Phi(x^*).
\]
All feasible directional derivatives are zero at $x^*$; it is a local
maximum along its feasible line. Nonnegative directional derivatives are
therefore not a sufficient local optimality condition.
\end{example}

The nonnegativity condition in \cite[Remark 3.5]{XNTW}
characterizes criticality by Proposition~\ref{Prop_stationary}, but
Example~\ref{Ex_stationary} shows that it is insufficient for local
minimality. Theorem~\ref{Th_growth} gives the corresponding necessary
and sufficient condition through strict positivity of all nonzero
feasible directional derivatives.

\subsection{Global Exact Recovery}\label{Sec_recovery}
Global exact recovery requires uniqueness on the whole feasible set.
Feasible displacements and the global descent cone provide equivalent
characterizations of this property.

\begin{definition}
A signal $x^*\in\Delta$ is \emph{exactly recovered} by minimizing
$\Phi$ over $\Delta$ if it is the unique global minimizer, that is,
$\argmin_{x\in\Delta}\Phi(x)=\{x^*\}$.
\end{definition}

\begin{theorem}\label{Th_recovery}
Let $x^*\in\Delta$. The following statements
are equivalent:
\begin{itemize}
\item[(i)] $x^*$ is exactly recovered by minimizing $\Phi$ over $\Delta$.
\item[(ii)] $\Phi(x^*+d)>\Phi(x^*)$ for every
$d\in\ker(A)\setminus\{0\}$.
\item[(iii)] $\Psi_{x^*}(x^*+d)>0$ for every
$d\in\ker(A)\setminus\{0\}$.
\item[(iv)] $\ker(A)\cap D_\Phi(x^*)=\{0\}$.
\end{itemize}
\end{theorem}
\begin{proof}
\emph{(i) $\Longleftrightarrow$ (ii).}
Every feasible point has the form $x^*+d$ with $d\in\ker(A)$.
Thus $x^*$ is the unique global minimizer if and only if (ii) holds.

\emph{(ii) $\Longleftrightarrow$ (iii).}
Since $Ax=b\ne0$ for every $x\in\Delta$, all feasible points
are nonzero. By \eqref{Eq_gap}, for every $d\in\ker(A)$,
\[
 \Psi_{x^*}(x^*+d)
 =\norm{x^*+d}_2\bigl(\Phi(x^*+d)-\Phi(x^*)\bigr).
\]
The positive norm factor makes (ii) and (iii) equivalent.

\emph{(ii) $\Longleftrightarrow$ (iv).}
If (ii) fails, there is a nonzero $d\in\ker(A)$ with
$\Phi(x^*+d)\le\Phi(x^*)$. Taking $t=1$ in
\eqref{Eq_descent} gives $d\in D_\Phi(x^*)$, so (iv) fails.
Conversely, if (iv) fails, take a nonzero
$d\in\ker(A)\cap D_\Phi(x^*)$. By \eqref{Eq_descent}, there
exists $t>0$ such that $\Phi(x^*+td)\le\Phi(x^*)$.
Since $td\in\ker(A)\setminus\{0\}$, this contradicts (ii).
\end{proof}

Combining Theorem~\ref{Th_recovery}(iv) with Lemma~\ref{Prop_descent}
gives an explicit test for exact recovery: for every nonzero
$d\in\ker(A)$, one must have
\[
 \Phi(d)\ge\Phi(x^*)
\]
and, for every $i\in S$ with $x_i^*d_i<0$,
\[
 \Phi\left(x^*-\frac{x_i^*}{d_i}d\right)>\Phi(x^*).
\]
Indeed, \eqref{Eq_no_parallel_feasible} excludes nonzero null-space
directions parallel to $x^*$. The two inequalities then exclude the
remaining sets in \eqref{Eq_cone}. The points in the second comparison
are feasible and hence nonzero, since their measurements equal $b\ne0$.
Thus, for each nonzero null-space direction, the test requires a
comparison of $\Phi(d)$ with $\Phi(x^*)$ and at most $|S|$ comparisons
at positive steps where a coordinate on $S$ vanishes.

\begin{remark}\label{Rem_recovery}
The strict comparison in Theorem~\ref{Th_recovery}(ii) appears,
with an additional quantifier over sparse signals, in
\cite[Eq.~(26)]{XNTW}. The descent-cone criterion is given for convex
objectives in \cite[Definition 2.7 and Fact 2.8]{ALMT} and recalled
for nuclear-norm recovery in \cite[Remark 3.1(c)]{LPB}.
The proof above shows that this set-theoretic criterion also applies
to the present nonconvex ratio without taking the closure of the cone.
\end{remark}

Combining the outer estimate in Lemma~\ref{Prop_outer} with
Theorem~\ref{Th_recovery}(iv) yields the following sufficient
conditions for global exact recovery.

\begin{corollary}\label{Cor_global_certificate}
Let $x^*\in\Delta$ and $S=\supp(x^*)$. If, for every nonzero
$d\in\ker(A)$,
$$
 \norm{d_{S^c}}_1>
 \sqrt{|S|}\,\operatorname{dist}(d,\operatorname{span}\{x^*\})
 ,
$$
then $x^*$ is exactly recovered by minimizing $\Phi$ over $\Delta$.
In particular, it suffices that
$\norm{d_{S^c}}_1>\sqrt{|S|}\norm d_2$ for every nonzero
$d\in\ker(A)$.
\end{corollary}

\begin{remark}\label{Rem_global_cones}
By Lemma~\ref{Prop_localcones} and the definitions of the descent cones,
\[
 \mathcal T_C(x^*)=\mathcal R_C(x^*)
 \subseteq D_\Phi(x^*)\subseteq\cl D_\Phi(x^*).
\]
Thus the condition
\[
 \ker(A)\cap\cl D_\Phi(x^*)=\{0\}
\]
is sufficient for both first-order growth and global exact recovery
by Theorems~\ref{Th_growth} and~\ref{Th_recovery}, but is not necessary
for either property.

To see this, take $x^*=(1,0)$ and $A=(1\ \ 0)$, so
$\Delta=\{(1,t):t\in\R\}$. For every feasible direction $d=(0,t)$,
$\Phi'(x^*;d)=|t|$, which gives first-order growth by
Theorem~\ref{Th_growth}. Moreover,
$\Phi(1,t)=(1+|t|)/\sqrt{1+t^2}>1=\Phi(x^*)$ for every $t\ne0$,
so $x^*$ is uniquely globally recovered. However, the directions
$d_k=(-1/k,1)$ belong to $D_\Phi(x^*)$, since
$x^*+kd_k=(0,k)$ has objective value one, while
\[
 d_k\longrightarrow(0,1)\in\ker(A)\setminus\{0\}.
\]
Hence taking the closure can introduce a nonzero null-space direction
even when both first-order growth and global exact recovery hold.
This differs from affine-constrained nuclear-norm minimization,
where the closed descent-cone condition is equivalent to first-order
growth \cite[Theorem 3.1]{LPB}.
\end{remark}

Although global exact recovery implies local exact recovery, the
converse fails even at a sharp local minimizer, as the following
example shows.

\begin{example}
\label{Ex_local}
Take
\[
 x^*=(2,1,0),\qquad
 A=\begin{pmatrix}1&-2&0\\0&1&1\end{pmatrix},\qquad
 b=\binom01,\qquad \ker(A)=\operatorname{span}\{d\},\quad d=(-2,-1,1).
\]
Substitution into \eqref{Eq_derivative} gives
\[
 \Phi'(x^*;td)
 =\frac{-3t+|t|}{\sqrt5}
 -\frac{3(-5t)}{5\sqrt5}
 =\frac{|t|}{\sqrt5}>0\qquad(t\ne0).
\]
Thus Theorem~\ref{Th_growth} shows that $x^*$ is a sharp local
minimizer. However,
\[
 x^*+d=(0,0,1),\qquad
 \Phi(x^*+d)=1<3/\sqrt5=\Phi(x^*).
\]
Hence $x^*$ is not globally exactly recovered. In particular,
$d\in D_\Phi(x^*)$ despite $\Phi'(x^*;d)>0$.
\end{example}

\section{Uniform Exact Recovery}\label{Sec_uniform}
The pointwise criteria of Section~\ref{Sec_pointwise} yield uniform
recovery conditions for $s$-sparse signals with nonzero measurements.
For a fixed matrix $A$ and an integer $1\le s\le n$, define
\[
 \mathcal X_s(A)=\{x\in\R^n:\norm{x}_0\le s,\ Ax\ne0\}.
\]
For each $x^*\in\mathcal X_s(A)$, set $b=Ax^*\ne0$ and denote
the corresponding feasible set by $\Delta$, as in \eqref{Eq_ratio}.
Uniform recovery requires the relevant
pointwise criterion to hold for every signal in this class. Bounding
the signal-dependent directional terms yields a uniform local
sufficient condition and quantitative growth estimates. The projection
bounds for the global descent cone yield uniform global sufficient
conditions. Optimality results establish the best coefficients within
the stated scalar families, and separating examples distinguish local
recovery from global recovery.

\subsection{Uniform Local Exact Recovery}\label{Sub_uniform_local}
The matrix $A$ provides uniform local exact recovery on
$\mathcal X_s(A)$ if every $x^*\in\mathcal X_s(A)$ is a strict local
minimizer of $\Phi$ on $\Delta$. The localization radius and growth modulus may depend on the signal.
Theorem~\ref{Th_growth} and Proposition~\ref{Th_local} give the
following criterion.

\begin{lemma}\label{Prop_uniform}
The matrix $A$ provides uniform local exact recovery on
$\mathcal X_s(A)$ if and only if the inequality
\begin{equation}\label{Eq_uniform_local}
 \norm{d_{S^c}}_1>
 \left|\ip{\sgn(x^*_S)}{d_S}
 -\frac{\norm{x^*}_1}{\norm{x^*}_2^2}\ip{x^*}{d}\right|
\end{equation}
holds for every $x^*\in\mathcal X_s(A)$ and every nonzero
$d\in\ker(A)$.
\end{lemma}
Condition~\eqref{Eq_uniform_local} depends on the signs and amplitudes
of $x^*$. Bounding the coefficients of $d_S$ uniformly gives a
sufficient condition depending only on $\ker(A)$ and $s$.

\begin{theorem}\label{Th_uniform_local}
Let $s$ be an integer with $2\le s\le n$, and set
\[
 c_s=\max\left\{1,\frac{\sqrt{s}-1}{2}\right\}.
\]
Suppose that for every nonzero $d\!\in\!\ker(A)$ and every index set
$T\!\subseteq\!\{1,\ldots,n\}$ with $1\le|T|\le s$,
\begin{equation}\label{Eq_local_certificate}
 \norm{d_{T^c}}_1>c_s\norm{d_T}_1.
\end{equation}
Then $A$ provides uniform local exact recovery on $\mathcal X_s(A)$,
and every signal in this class is a sharp and tilt-stable local
minimizer of its associated problem. For $2\le s\le9$,
condition~\eqref{Eq_local_certificate} is precisely the
$\ell_1$ null-space property of order $s$ \cite{DH,GN}.
\end{theorem}
\begin{proof}
Fix $x^*\in\mathcal X_s(A)$ and write $S=\supp(x^*)$ and $r=|S|$.
Since $1\le r\le s$, condition~\eqref{Eq_local_certificate}
applies with $T=S$.
If $r=1$, then
$\sgn(x^*_S)-\norm{x^*}_1x^*_S/\norm{x^*}_2^2=0$,
and \eqref{Eq_local_certificate} gives
$\norm{d_{S^c}}_1>0$ for every nonzero $d\in\ker(A)$.
Thus \eqref{Eq_uniform_local} holds.

Suppose $r\ge2$ and set $u=|x^*|/\norm{x^*}_2\in\R^n$,
where $|x^*|$ denotes the vector of componentwise absolute values.
Then $u_i>0$ for $i\in S$, $u_{S^c}=0$, and $\norm u_2=1$.
Fix $i\in S$.
Applying the Cauchy--Schwarz inequality to the remaining $r-1$ components gives
\[
 \sum_{j\in S\setminus\{i\}}u_j
 \le\sqrt{r-1}\left(\sum_{j\in S\setminus\{i\}}u_j^2\right)^{1/2}
 =\sqrt{r-1}\sqrt{1-u_i^2}.
\]
Since $\norm u_1=u_i+\sum_{j\in S\setminus\{i\}}u_j$, it follows that
\begin{equation}\label{Eq_support_coefficient_bound}
 0\le u_i\norm u_1
 \le u_i^2+\sqrt{r-1}\,u_i\sqrt{1-u_i^2}
 .
\end{equation}
To bound the right-hand side, set $t=u_i^2\in[0,1]$.
Another application of the Cauchy--Schwarz inequality gives
\[
\begin{aligned}
 2t-1+2\sqrt{r-1}\sqrt{t(1-t)}
 \le\sqrt{1+(r-1)}\sqrt{(2t-1)^2+4t(1-t)}=\sqrt{r}.
\end{aligned}
\]
Rearranging gives
\[
 t+\sqrt{r-1}\sqrt{t(1-t)}\le\frac{1+\sqrt{r}}{2}.
\]
Combining this with \eqref{Eq_support_coefficient_bound} shows that
$0\le u_i\norm u_1\le(1+\sqrt{r})/2$.
Consequently,
\[
 \left|1-u_i\norm u_1\right|
 \le\max\left\{1,\frac{\sqrt{r}-1}{2}\right\}\le c_s.
\]
It follows that, for every nonzero $d\in\ker(A)$,
\begin{equation}\label{Eq_uniform_linear_bound}
 \begin{aligned}
 \left|\ip{\sgn(x^*_S)}{d_S}
 -\frac{\norm{x^*}_1}{\norm{x^*}_2^2}\ip{x^*}{d}\right|
 \le
 \sum_{i\in S}|d_i|\,\left|1-u_i\norm u_1\right|\le c_s\norm{d_S}_1<\norm{d_{S^c}}_1.
 \end{aligned}
\end{equation}
The desired conclusion is obtained by
Lemma~\ref{Prop_uniform}.
Finally, $c_s=1$ for $2\le s\le9$.
\end{proof}

\begin{remark}\label{Rem_uniform_scope}
For $s=1$, uniform local exact recovery on $\mathcal X_1(A)$
holds without a null-space assumption: the support term in
\eqref{Eq_uniform_local} vanishes, and a nonzero feasible null-space
direction cannot be supported on the single nonzero coordinate of
$x^*$.
\end{remark}

A positive margin in the local null-space condition gives explicit
growth moduli,
localization radii, and bounds on linear perturbations.
\begin{corollary}\label{Prop_uniform_bounds}
Suppose that the assumptions of Theorem~\ref{Th_uniform_local} hold
and $\ker(A)\ne\{0\}$. Define the margin in \eqref{Eq_local_certificate} by
\begin{equation}\label{Eq_margin}
 \eta:=
 \min_{\substack{d\in\ker(A),\ \norm d_2=1\\1\le|T|\le s}}
 \bigl(\norm{d_{T^c}}_1-c_s\norm{d_T}_1\bigr).
\end{equation}
Then $\eta>0$. For each $x^*\in\mathcal X_s(A)$, let
\[
 \rho=\min\left\{
       \frac12\min_{i\in S}|x_i^*|,
       \frac{\eta\norm{x^*}_2}{\sqrt{s}}\right\},
 \qquad
 \kappa=\frac{\eta}{3\norm{x^*}_2}.
\]
The signal $x^*$ satisfies first-order growth with modulus $\kappa$
and radius $\rho$, namely,
\begin{equation}\label{Eq_uniform_growth}
 \Phi(x)\ge\Phi(x^*)+\kappa\norm{x-x^*}_2,
\end{equation}
for every $x\in\Delta$ with $\norm{x-x^*}_2\le\rho$.
In particular, $x^*$ is uniquely recovered on this feasible neighborhood
and remains its unique minimizer under every linear perturbation
$v\in\R^n$ with $\norm v_2<\kappa$:
\[
 \argmin_{\substack{x\in\Delta\\\norm{x-x^*}_2\le\rho}}
 \{\Phi(x)-\ip v{x-x^*}\}=\{x^*\}.
\]
\end{corollary}
\begin{proof}
For each index set $T$ with $1\le|T|\le s$, the function
$d\mapsto\norm{d_{T^c}}_1-c_s\norm{d_T}_1$ is continuous and
strictly positive on the compact unit sphere of $\ker(A)$.
There are finitely many such index sets, so \eqref{Eq_margin}
gives $\eta>0$.

Fix $x^*\in\mathcal X_s(A)$ and set
$r=\norm{x^*}_2$ and $\omega=\Phi(x^*)$.
By \eqref{Eq_uniform_linear_bound} and the definition of $\eta$
with $T=S$, for every $d\in\ker(A)$,
\begin{equation}\label{Eq_margin_linear_bound}
 \norm{d_{S^c}}_1+\ip{\sgn(x^*_S)}{d_S}
 -\frac{\norm{x^*}_1}{r^2}\ip{x^*}{d}
 \ge\norm{d_{S^c}}_1-c_s\norm{d_S}_1
 \ge\eta\norm d_2.
\end{equation}
The last inequality follows by applying \eqref{Eq_margin} to
$d/\norm d_2$ when $d\ne0$, and is immediate when $d=0$.

Now take any $d\in\ker(A)$ with $\norm d_2\le\rho$.
Since $\rho\le\tfrac12\min_{i\in S}|x_i^*|$, the coordinates
of $x^*+d$ on $S$ retain their signs, and \eqref{diff_l1} gives
the exact expansion of its $\ell_1$ norm. Notice that
\[
 \norm{x^*+d}_2
 =r\sqrt{1+\frac{2\ip{x^*}{d}+\norm d_2^2}{r^2}}.
\]
 Applying
$\sqrt{1+t}\le1+t/2$ for $t\ge-1$ therefore yields
\begin{equation}\label{Eq_margin_norm_bound}
 \norm{x^*+d}_2
 \le r+\frac{\ip{x^*}{d}}r+\frac{\norm d_2^2}{2r}.
\end{equation}
Subtracting $\omega$ times \eqref{Eq_margin_norm_bound} from
\eqref{diff_l1}, and using $\omega r=\norm{x^*}_1$ and
\eqref{Eq_margin_linear_bound}, gives
\[
 \norm{x^*+d}_1-\omega\norm{x^*+d}_2
 \ge\eta\norm d_2-\frac{\omega}{2r}\norm d_2^2
 \ge\frac{\eta}{2}\norm d_2.
\]
Here the last inequality follows from $\omega\le\sqrt{s}$ and
$\norm d_2\le\rho\le\eta r/\sqrt{s}$.

Also, $\rho\le r/2$, so the triangle inequality gives
$0<r/2\le\norm{x^*+d}_2\le3r/2$.
Hence, 
\[
\begin{aligned}
 \Phi(x^*+d)-\Phi(x^*)
 =\frac{\norm{x^*+d}_1-\omega\norm{x^*+d}_2}
          {\norm{x^*+d}_2}\ge\frac{(\eta/2)\norm d_2}{3r/2}
 =\kappa\norm d_2.
\end{aligned}
\]
Taking $x=x^*+d$ proves \eqref{Eq_uniform_growth}.
Uniqueness and the perturbation bound follow from
Proposition~\ref{Th_local}.
\end{proof}
The growth modulus cannot be chosen uniformly over all signal scales.
To see this, fix $x^*\in\mathcal X_s(A)$ and a unit vector
$d\in\ker(A)$ under the assumptions of
Corollary~\ref{Prop_uniform_bounds}. For every $t>0$, the scaled signal $t x^*$
also belongs to $\mathcal X_s(A)$, and scale invariance gives
\[
 \Phi'(t x^*;d)=t^{-1}\Phi'(x^*;d)\longrightarrow0
 \qquad(t\to\infty).
\]
Any local growth modulus at $t x^*$ is bounded above by
$\Phi'(t x^*;d)$, as follows by applying the growth inequality
along this unit feasible direction and taking the directional
derivative. Hence no common positive growth modulus exists over
the unnormalized class.

On the normalized class
$\norm{x^*}_2=1$, Corollary~\ref{Prop_uniform_bounds} gives the
common growth modulus and perturbation threshold $\eta/3$.
If these signals also satisfy $|x_i^*|\ge a$
for every $i\in\supp(x^*)$, with a fixed $a>0$, then
$\rho_0=\min\{a/2,\eta/\sqrt{s}\}$ is a common localization radius.

\begin{remark}\label{Rem_local_comparison}
The following four sufficient conditions for uniform local exact
recovery are compared for $s\ge2$. Each must hold for every nonzero
$d\in\ker(A)$ and every index set $T$ with $1\le|T|\le s$:
\[
 \begin{array}{ll}
 \text{Rahimi et al.~\cite[Theorem 2.2]{RWDL}:}
 &\norm{d_{T^c}}_1\ge(s+1)\norm{d_T}_1,\\[1mm]
 \text{Xu et al.~\cite[Corollary 3.2]{XNTW}:}
 &\norm{d_{T^c}}_1>(\sqrt{s}+2)\norm{d_T}_1,\;  s>6,\\[1mm]
 \text{Yu and Huang~\cite[Corollary 3.3]{YH}:}
 &\norm{d_{T^c}}_1>\dfrac{\sqrt{s}+3}{2}\norm{d_T}_1,\\[1mm]
 \text{Theorem~\ref{Th_uniform_local}:}
 &\norm{d_{T^c}}_1>\max\left\{1,\dfrac{\sqrt{s}-1}{2}\right\}\norm{d_T}_1.
 \end{array}
\]
The third condition follows by specializing
\cite[Corollary 3.3]{YH} to $p=1$ and $q=2$, for which
$K_{1,2,s}=(1+\sqrt{s})/2$ and the required coefficient is
$1+K_{1,2,s}=(3+\sqrt{s})/2$.

In each stated range, $c_s$ is smaller than the corresponding earlier
coefficient, so all three conditions imply \eqref{Eq_local_certificate}.
Example~\ref{Ex_uniform_comparison} shows that
\eqref{Eq_local_certificate} can hold when all three displayed
earlier conditions fail.

\end{remark}

\begin{example}\label{Ex_uniform_comparison}
Take $s=7$ and
\[
 A=[I_{14}\ \ -\mathbf1_{14}]\in\R^{14\times15},
 \qquad \ker(A)=\operatorname{span}\{\mathbf1_{15}\}.
\]
Since $c_7=1$, every nonzero $d=t\mathbf1_{15}$ and every
$|T|\le7$ satisfy
\[
 \norm{d_{T^c}}_1\ge8|t|>7|t|\ge\norm{d_T}_1.
\]
Thus Theorem~\ref{Th_uniform_local} guarantees uniform local
exact recovery, first-order growth, and tilt stability for all signals in
$\mathcal X_7(A)$.
For $|T|=7$, the ratio
$\norm{d_{T^c}}_1/\norm{d_T}_1=8/7$ is smaller than each of
$8$, $\sqrt7+2$, and $(\sqrt7+3)/2$.
Thus this matrix satisfies none of the three earlier conditions. 
\end{example}

For each fixed $s\ge2$, the coefficient $c_s$ cannot be reduced
in \eqref{Eq_local_certificate} if the condition is to guarantee
uniform local exact recovery for all dimensions $n$ and measurement matrices $A$.
\begin{proposition}\label{Th_coefficient}
Let $s\ge2$ be an integer, and let $c_s$ be defined as in
Theorem~\ref{Th_uniform_local}. For every $0\le c<c_s$, there exist
an integer $n>s$, a full row rank matrix
$A\in\R^{(n-1)\times n}$, and an exactly $s$-sparse signal
$x^*\in\mathcal X_s(A)$ such that
\[
 \norm{d_{T^c}}_1>c\norm{d_T}_1
\]
for every nonzero $d\in\ker(A)$ and every index set $T$ with
$|T|\le s$,
but $x^*$ is not a local minimizer of $\Phi$ on $\Delta$.
Consequently, $c_s$ is the smallest nonnegative coefficient for which
\eqref{Eq_local_certificate} guarantees uniform local exact recovery
independently of the dimension $n$ and the measurement matrix $A$.
\end{proposition}
\begin{proof}
Fix $0\le c<c_s$. Choose a positive vector $u\in\R^s$ such that
\[
 \norm u_2=1,\qquad \left|1-u_1\norm u_1\right|>c.
\]
Such a vector can be chosen as follows. If $c_s=1$, let
$u_1\in(0,1)$ tend to zero and set
$u_2=\cdots=u_s=\sqrt{(1-u_1^2)/(s-1)}$.
Then $1-u_1\norm u_1\to1$, so the required strict inequality
holds for sufficiently small $u_1$.
If $c_s>1$, choose the positive entries of $u$ so that
\[
 u_1^2=\frac{1+s^{-1/2}}2,\qquad
 u_2^2=\cdots=u_s^2=\frac{1-u_1^2}{s-1}.
\]
For this choice, $\norm u_2=1$ and
\[
 u_1\norm u_1=\frac{1+\sqrt{s}}2,\qquad
 |1-u_1\norm u_1|=\frac{\sqrt{s}-1}{2}=c_s>c.
\]

Choose $M$ with $c\!<\!M\!<\!|1-u_1\norm u_1|$. Then an integer $N\!\ge\! s$ can be chosen so that
\[
 N>\max\left\{M,\frac{(s-1)(1+c)M}{M-c}\right\}.
\]
Setting $\varepsilon=M/N$ then gives $0\!<\!\varepsilon\!<\!1$ and
\begin{align}\label{N_choose}
 M-(s-1)\varepsilon>c\bigl(1+(s-1)\varepsilon\bigr).
\end{align}
Set $n:=s+N$ and 
\[
 x^*:=(u,0_N),\quad
 d:=\bigl(-\sgn(1-u_1\norm u_1),\,0_{s-1},
                  \varepsilon\mathbf1_N\bigr),
                  \quad A:=\left[-\frac{d_{\{2,\ldots,n\}}}{d_1}\ \ I_{n-1}\right].
\]
Obviously, $A$ has full row rank and $Ad=0$.
So, $\ker(A)=\operatorname{span}\{d\}$. Since
$x^*_2=u_2>0$ and $d_2=0$, we have $x^*\notin\operatorname{span}\{d\}$.
Thus $Ax^*\ne0$, and hence $x^*\in\mathcal X_s(A)$ with $\|x^*\|_0=s$.

To verify the null-space inequality, note that the $s$ largest
absolute entries of $d$ consist of $1$ and $s-1$ copies of
$\varepsilon$. Thus, for every index set $T$ with $|T|\le s$,
\[
 \norm{d_T}_1\le1+(s-1)\varepsilon.
\]
Since $\norm d_1=1+M$, inequality~\eqref{N_choose} gives
\[
 \norm{d_{T^c}}_1=1+M-\norm{d_T}_1
 \ge M-(s-1)\varepsilon
 >c\bigl(1+(s-1)\varepsilon\bigr)\ge c\norm{d_T}_1.
\]
Since $\ker(A)\!=\!\operatorname{span}\{d\}$, $\norm{d_{T^c}}_1>c\norm{d_T}_1$ for every nonzero $d\in\ker(A)$.

It remains to show that $x^*$ is not a local minimizer.
Here $S=\{1,\ldots,s\}$, $\norm{x^*}_2=\norm{u}_2=1$, and
$\norm{d_{S^c}}_1=N\varepsilon=M$.
Since the only nonzero component of $d_S$ is $d_1$,
\eqref{Eq_derivative} gives
\[
\begin{aligned}
 \Phi'(x^*;d)
 =(1-u_1\norm u_1)d_1+M=-\left|1-u_1\norm u_1\right|+M<0.
\end{aligned}
\]
Thus $\Phi(x^*+td)<\Phi(x^*)$ for all sufficiently small $t>0$.
Since $d\in\ker(A)$, these points are feasible, so $x^*$ is not a
local minimizer.
\end{proof}
The optimality in Proposition~\ref{Th_coefficient} concerns the scalar
family \eqref{Eq_local_certificate}. For a fixed matrix, a condition
that uses the individual entries of $d_S$ or the amplitudes of $x^*$
may be less restrictive. Moreover, \eqref{Eq_local_certificate} does
not guarantee uniform global exact recovery, as the next example shows.
\begin{example}\label{Ex_uniform_local_global}
Take $s=7$ and set
\[
 d=(\mathbf1_7,\,10,\,\tfrac12\mathbf1_{32})\in\R^{40},
 \qquad
 A=[I_{39}\ \ -2d_{\{1,\ldots,39\}}],
 \qquad
 x^*=(-\mathbf1_7,0_{33}).
\]
Then $\ker(A)=\operatorname{span}\{d\}$ and $x^*\in\mathcal X_7(A)$.
Since $\norm d_1=33$ and the seven largest absolute entries of $d$
sum to $16$, every index set $T$ with $|T|\le7$ satisfies
\[
 \norm{d_{T^c}}_1-\norm{d_T}_1
 =33-2\norm{d_T}_1\ge1.
\]
By homogeneity and $c_7=1$, $A$ satisfies
\eqref{Eq_local_certificate} and hence provides uniform local exact
recovery on $\mathcal X_7(A)$ by Theorem~\ref{Th_uniform_local}.
However,
\[
 x^*+d=(0_7,\,10,\,\tfrac12\mathbf1_{32}),\qquad
 \Phi(x^*+d)=\frac{26}{\sqrt{108}}<\sqrt7=\Phi(x^*).
\]
Thus $x^*$ is not a global minimizer, so uniform global exact recovery
on $\mathcal X_7(A)$ fails.
\end{example}

\subsection{Uniform Global Exact Recovery}\label{Sub_uniform_global}
The matrix $A$ provides uniform global exact recovery on
$\mathcal X_s(A)$ if every $x^*\in\mathcal X_s(A)$ is the unique
global minimizer of $\Phi$ on $\Delta$.
Applying Theorem~\ref{Th_recovery} to each signal gives the following
characterization.

\begin{lemma}\label{Prop_uniform_global}
The matrix $A$ provides uniform global exact recovery on
$\mathcal X_s(A)$ if and only if the inequality
\begin{equation}\label{Eq_uniform_global}
 \norm{x^*+d}_1>
 \Phi(x^*)\norm{x^*+d}_2
\end{equation}
holds for every $x^*\in\mathcal X_s(A)$ and every nonzero
$d\in\ker(A)$. Equivalently,
\[
 \ker(A)\cap D_\Phi(x^*)=\{0\}
\]
for every $x^*\in\mathcal X_s(A)$.
\end{lemma}
Lemma~\ref{Prop_outer} replaces the comparisons over all
$x^*\in\mathcal X_s(A)$ in \eqref{Eq_uniform_global} by a sufficient
condition on $\ker(A)$ and $s$.

\begin{theorem}\label{Th_uniform_global}
Suppose that for every nonzero $d\in\ker(A)$ and every index set
$T\subseteq\{1,\ldots,n\}$ with $|T|\le s$,
\begin{equation}\label{Eq_global_certificate}
 \norm{d_{T^c}}_1>\sqrt{s}\norm d_2.
\end{equation}
Then $A$ provides uniform global exact recovery on $\mathcal X_s(A)$.
In particular, \eqref{Eq_global_certificate} holds if
\begin{equation}\label{Eq_global_ratio}
 \inf_{0\ne d\in\ker(A)}
       \Phi(d)>2\sqrt{s},
\end{equation}
where the infimum is interpreted as $+\infty$
when $\ker(A)=\{0\}$.
\end{theorem}
\begin{proof}
Fix $x^*\in\mathcal X_s(A)$ and set $S=\supp(x^*)$.
For every $d\in D_\Phi(x^*)$, Lemma~\ref{Prop_outer} and $|S|\le s$ give
\begin{equation*} \norm{d_{S^c}}_1
 \le\sqrt{|S|}\norm d_2\le\sqrt{s}\norm d_2.
\end{equation*}
Activating condition~\eqref{Eq_global_certificate} by taking $T\!=\!S$ leads to
$\ker(A)\cap D_\Phi(x^*)\!=\!\{0\}$ for every
$x^*\in\mathcal X_s(A)$, and then Lemma~\ref{Prop_uniform_global}
yields uniform global exact recovery.

To prove the final assertion, suppose \eqref{Eq_global_ratio} holds.
For every nonzero $d\in\ker(A)$ and every index set $T$ with
$|T|\le s$, the Cauchy--Schwarz inequality gives
$\norm{d_T}_1\le\sqrt{|T|}\norm{d_T}_2\le\sqrt{s}\norm d_2$.
Hence,
\(
 \norm{d_{T^c}}_1\!=\!\norm d_1-\norm{d_T}_1
 \!>\!2\sqrt{s}\norm d_2-\sqrt{s}\norm d_2
 \!=\!\sqrt{s}\norm d_2.
\)
Thus \eqref{Eq_global_ratio} implies \eqref{Eq_global_certificate}.
\end{proof}

\begin{remark}\label{Rem_global_certificate}
The sufficient condition in \cite[Eq.~(31)]{XNTW},
specialized to $q=1$, requires
\begin{equation}\label{Eq_global_old}
 \norm{d_{T^c}}_1-\norm{d_T}_1>\sqrt{s}\norm d_2
\end{equation}
for every nonzero $d\in\ker(A)$ and every index set $T$ with
$|T|\le s$. This condition implies \eqref{Eq_global_certificate}.
The scalar bound \eqref{Eq_global_ratio} also improves the coefficient
three in \cite[Theorem 4.1]{XNTW} to two.
Example~\ref{Ex_global_certificate} shows that
\eqref{Eq_global_certificate} can hold when both
\eqref{Eq_global_old} and \eqref{Eq_global_ratio} fail.
\end{remark}

The coefficient two in \eqref{Eq_global_ratio} is optimal when both
$s$ and $n$ may vary.
\begin{proposition}\label{Prop_global_coefficient}
For every $0\le c<2$, there exist integers $2\le s<n$, a full row
rank matrix $A\in\R^{(n-1)\times n}$, and an exactly $s$-sparse
signal $x^*\in\mathcal X_s(A)$ such that
\[
 \inf_{0\ne d\in\ker(A)}\Phi(d)>c\sqrt{s},
\]
but $x^*$ is not a global minimizer of $\Phi$ on $\Delta$.
Consequently, the coefficient two in
\eqref{Eq_global_ratio} cannot be decreased uniformly over all $s$ and $n$.
\end{proposition}
\begin{proof}
Fix $0\le c<2$. Since
\[
 \frac{2k-1}{\sqrt{(k+1)(k+1-1/k)}}\longrightarrow2
 \qquad(k\to\infty),
\]
an integer $k\ge2$ can be chosen sufficiently large that
\begin{equation}\label{Eq_global_k_choice}
 \frac{2k-1}{\sqrt{(k+1)(k+1-1/k)}}>c.
\end{equation}
Set $s=k+1$ and $n=k^2+1$, and define
\[
 x^*=(k,\mathbf1_k,0_{k(k-1)}),\qquad
 d=(0,-\mathbf1_k,k^{-1}\mathbf1_{k(k-1)}),
\]
and
\[
 A=\begin{pmatrix}
 1&0&0&0\\
 0&-\mathbf1_{k-1}&I_{k-1}&0\\
 0&k^{-1}\mathbf1_{k(k-1)}&0&I_{k(k-1)}
 \end{pmatrix}\in\R^{(n-1)\times n}.
\]
Notice that $A$ has
full row rank and $Ad=0$. Hence,
$\ker(A)=\operatorname{span}\{d\}$.
Moreover, $\norm{x^*}_0=s$ and the first component of $Ax^*$ is
$k>0$, so $x^*\in\mathcal X_s(A)$.

Note that $\norm d_1=2k-1$ and
$\norm d_2^2=k+1-1/k$. By \eqref{Eq_global_k_choice}, it holds that
\[
 \inf_{0\ne h\in\ker(A)}\Phi(h)
 =\Phi(d)=\frac{2k-1}{\sqrt{k+1-1/k}}>c\sqrt{k+1}=c\sqrt{s}.
\]

Since $x^*+d\!=\!(k,0_k,k^{-1}\mathbf1_{k(k-1)})\!\in\! \Delta$ and
\[
 \Phi(x^*+d)=\frac{2k-1}{\sqrt{k^2+1-1/k}}
 <2-\frac1k<\frac{2k}{\sqrt{k^2+k}}=\Phi(x^*).
\]
The strict inequalities follow from $k^2+1-1/k>k^2$ and
$\sqrt{k^2+k}<k+1/2<2k/(2-1/k)$, respectively.
Thus $x^*$ is not a global minimizer of $\Phi$ on $\Delta$.
\end{proof}
The best coefficient for a fixed $s$ remains undetermined.
The next example shows that \eqref{Eq_global_certificate} can hold
when both \eqref{Eq_global_old} and \eqref{Eq_global_ratio} fail.

\begin{example}\label{Ex_global_certificate}
Let $s=2$ and $A=[I_6\ \ -\mathbf1_6]\in\R^{6\times7}$. Then
$\ker(A)=\operatorname{span}\{\mathbf1_7\}$.
Every nonzero $d\in\ker(A)$ has the form $d=a\mathbf1_7$ with
$a\ne0$, so $\norm d_2=\sqrt7|a|$. For every index set $T$ with
$|T|\le2$,
\[
 \norm{d_{T^c}}_1=(7-|T|)|a|\ge5|a|>
 \sqrt{14}|a|=\sqrt{s}\norm d_2.
\]
Thus \eqref{Eq_global_certificate} holds, and
Theorem~\ref{Th_uniform_global} guarantees uniform global exact
recovery on $\mathcal X_2(A)$.

However, for any $T$ with $|T|=2$,
\[
 \norm{d_{T^c}}_1-\norm{d_T}_1=3|a|
 <\sqrt{14}|a|=\sqrt{s}\norm d_2,
\]
so \eqref{Eq_global_old} fails. Moreover,
\[
 \inf_{0\ne d\in\ker(A)}\Phi(d)
 =\sqrt7<2\sqrt2=2\sqrt{s},
\]
so \eqref{Eq_global_ratio} also fails. Hence neither condition is
necessary for uniform global exact recovery.
\end{example}

Uniform global exact recovery implies uniform local exact recovery,
but the converse fails, as shown by Example~\ref{Ex_uniform_local_global}.
By Proposition~\ref{Th_local}, each globally recovered signal also
has first-order growth and tilt stability. The corresponding growth
modulus and localization radius may depend on the signal.

\section*{Conclusions}

The variational geometric analysis establishes necessary and
sufficient conditions for local and global exact recovery in the
affine-constrained ratio model with nonzero data. The local
characterizations show that every local minimizer is sharp and
tilt stable. In particular, it remains the unique minimizer in a
sufficiently small feasible neighborhood under sufficiently small
linear objective perturbations. Global recovery requires additional
comparisons, captured by the global descent cone. The examples
demonstrate that even uniform local exact recovery does not ensure
uniform global exact recovery.

The pointwise analysis also yields improved sufficient null-space
conditions for uniform recovery. The coefficient optimality results
identify the limits of the considered scalar conditions, for each
fixed sparsity order in the local case and over all sparsity orders
and dimensions in the global case. A positive margin in the local
condition further provides quantitative bounds on growth,
localization, and admissible linear objective perturbations.
Together, these results connect recovery at a fixed signal with
uniform guarantees, while distinguishing sufficient scalar bounds
from exact recovery characterizations.



\end{document}